\documentclass[12pt,reqno]{amsart}
\usepackage{amsmath,amssymb,amscd,verbatim,color}

\usepackage[backref]{hyperref}

\usepackage{amsfonts}
\usepackage[top=35mm, bottom=35mm, left=30mm, right=30mm]{geometry}

\usepackage{xcolor}

\theoremstyle{plain}
\newtheorem{thm}{Theorem}[section]

\newtheorem{lem}[thm]{Lemma}
\newtheorem{prop}[thm]{Proposition}

\theoremstyle{definition}
\newtheorem{defn}[thm]{Definition}

\newtheorem{rem}[thm]{Remark}

\numberwithin{equation}{section}

\newcommand{\mb}{\mathbb}
\newcommand{\mc}{\mathcal}

\def \a{\alpha} \def \b{\beta} \def \g{\gamma} \def \d{\delta}
\def \t{\theta}   \def \e{\epsilon} 
\def \s{\sigma} \def \l{\lambda}  
\def \O{\Omega}

\def \L{\Lambda}

\allowdisplaybreaks

\renewcommand*{\backref}[1]{}
\renewcommand*{\backrefalt}[4]{\quad \tiny
  \ifcase #1 (\textbf{NOT CITED.})%
  \or    (Cited on page~#2.)%
  \else   (Cited on pages~#2.)%
  \fi}

\makeatletter

\def\MRbibitem{\@ifnextchar[\my@lbibitem\my@bibitem}

\def\mybiblabel#1#2{\@biblabel{{\hyperref{http://www.ams.org/mathscinet-getitem?mr=#1}{}{}{#2}}}}

\def\myhyperanchor#1{\Hy@raisedlink{\hyper@anchorstart{cite.#1}\hyper@anchorend}}

\def\my@lbibitem[#1]#2#3#4\par{%
  \item[\mybiblabel{#2}{#1}\myhyperanchor{#3}\hfill]#4%
  \@ifundefined{ifbackrefparscan}{}{\BR@backref{#3}}%
  \if@filesw{\let\protect\noexpand\immediate
    \write\@auxout{\string\bibcite{#3}{#1}}}\fi\ignorespaces%
}

\def\my@bibitem#1#2#3\par{%
  \refstepcounter\@listctr
  \item[\mybiblabel{#1}{\the\value\@listctr}\myhyperanchor{#2}\hfill]#3%
  \@ifundefined{ifbackrefparscan}{}{\BR@backref{#2}}%
  \if@filesw\immediate\write\@auxout
    {\string\bibcite{#2}{\the\value\@listctr}}\fi\ignorespaces%
}

\makeatother

\begin{document}

\author{Wen Huang} \address[Wen Huang] {Wu Wen-Tsun Key Laboratory of Mathematics, USTC, Chinese Academy of Sciences and Department of Mathematics \\
 University of Science and Technology of China,\\
  Hefei, Anhui, China}
\email[W. Huang]{wenh@mail.ustc.edu.cn}

 \author{Zeng Lian} \address[Zeng Lian] {College of Mathematical Sciences\\ Sichuan University\\
    Chengdu, Sichuan, 610016, China} \email[Z. Lian]{lianzeng@scu.edu.cn, zenglian@gmail.com}

\author{Xiao Ma} \address[Xiao Ma] {  Wu Wen-Tsun Key Laboratory of Mathematics, USTC, Chinese Academy of Sciences and Department of Mathematics \\
 University of Science and Technology of China \\
  Hefei, Anhui, China}
\email[X. Ma]{xiaoma@ustc.edu.cn}

\author{Leiye Xu} \address[Leiye Xu] { Wu Wen-Tsun Key Laboratory of Mathematics, USTC, Chinese Academy of Sciences and Department of Mathematics \\
 University of Science and Technology of China \\
  Hefei, Anhui, China}
\email[L. Xu]{leoasa@mail.ustc.edu.cn}

\author{Yiwei Zhang} \address[Yiwei Zhang] { School of Mathematics and Statistics, Center for Mathematical Sciences, Hubei Key Laboratory of
Engineering Modeling and Scientific Computing,\ Hua-Zhong University of Sciences and Technology \\
Wuhan 430074, China}
\email[Y. Zhang]{yiweizhang@hust.edu.cn, yiweizhang831129@gmail.com}

\title[]{Ergodic optimization theory for  Axiom A flows}

\thanks{Huang is partially supported by NSF of China (11431012,11731003). Lian is partially supported by NSF of China (11725105,11671279). Xu is partially supported by NSF of China (11801538, 11871188). Zhang is partially supported by NSF of China (117010200,11871262).}

\begin{abstract}
In this article, we consider the weighted ergodic optimization problem Axiom A attractors  of a $C^2$ flow on a compact smooth manifold. The main result obtained in this paper  is  that for a generic observable from function space $\mc C^{0,\a}$ ($\a\in(0,1]$) or $\mc C^1$  the minimizing measure is unique and is supported on a periodic orbit.
\end{abstract}

\maketitle

\parskip 0.4cm
\section{Introduction}

\noindent {\bf Context and motivation.}
Ergodic optimization theory focuses on the ergodic measures on which a given observable taking an extreme ergodic average (maximum or  minimum), which
 has strong connection with other fields, such as Anbry-Mather theory \cite{Contreras_Mane,Mane,CIPP} in Lagrangian Mechanics; ground state theory \cite{BLL} in thermodynamics formalism and multifractal analysis; and controlling chaos \cite{HO1,OGY90,SGOY93} in control theory.

 In this paper, we study the typical optimization problem in weighted ergodic optimization theory for  Axiom A attractors of a $C^2$ flow on a compact smooth manifold.
 For discrete time case, ergodic optimization theory has been developed broadly. Among them, Yuan and Hunt proposed an open problem in \cite[Conjecture 1.1]{YH} on 1999,
which provides a mathematical mechanism on Hunt and Ott's experimental and heuristic results in \cite{HO2,HO3} and becomes one of the fundamental questions raised in the field of ergodic optimization theory. A more general form of Yuan and Hunt's conjecture is now called the Typical Periodic Optimization Conjecture, and has attracted sustained attentions and yielded considerable results, for instances \cite{BZ,Bousch_poisson,Bousch_Walters,Bousch_note,Contreras_EO,CLT,Morris_entropy,QS}. For a more comprehensive survey for the classical ergodic optimization theory, we refer the reader to Jenkinson \cite{Jenkinson_survey,Jenkinson_newsurvey}, to Bochi \cite{Bochi}, to Baraviera, Leplaideur, Lopes \cite{BLL}, and to Garibaldi \cite{G} for a historical perspective of the development in this area. In our recent paper \cite{HLMXZ}, we extend the applicability of the theory both to a broader class of systems  and to  a broader class of observables, which leads to a positive answer to Yuan and Hunt's conjecture for $C^1$-observable case. To our knowledge, because of difficulties appears on both conceptual level and technical level, there is no existing result of ergodic optimization theory for flows so far, which make the results obtained in the present paper the first achievement on flows towards  ergodic optimization theory.

On the other hand, as mentioned in \cite{HLMXZ}, the reason of adding the non-constant weight $\psi$ mainly lies in the studies on the zero temperature limit (or ground state) of the $(u,\psi)$-weighted equilibrium state for thermodynamics formalism (for more details, we refer readers to works \cite{BF,BCW,FH}).

\noindent{\bf Summary of the results.} To avoid unnecessary complexity, we  only introduce the result in the framework of standard ergodic optimization theory. Consider a $C^2$  flow $\Phi$ on a compact smooth manifold $M$. Let $\L$ be an Axiom A attractor of $\Phi$ (detailed definition is given by Definition \ref{D:AxiomAAttractor}).  For a given observable $u:M\to \mb R$, the ergodic averages of $u$ on $\L$  is defined by the integration of $u$ with respect to $\Phi|_\L$-ergodic measures, and the $u$-maximizing (or minimizing) measure is the measure with respect to which the ergodic average of $u$ takes maximum (or minimum) value. As a consequence of the main result (Theorem \ref{T:MainResult}) of the present paper, we have the following result:

\noindent{\bf Theorem A:} Let $(\L,\Phi)$ be an Axiom A attractor on a compact smooth Riemannian manifold $M$, then for a generic observable $u$ in function space $\mc C^{0,\a}(M)$ or $\mc C^1(M)$, the $u$-maximizing (or $u$-minimizing) measure is unique and is supported  on a periodic orbit.\\

\noindent{\bf Remarks on techniques of the proof.} It seems that the proof given in \cite{HLMXZ} provides a more general mechanism in the study on ergodic optimization problems, which also shed a light on the case of flows for sure.   However, the results of the present paper depend crucially on the continuous time nature of the system; that is to say, they do not follow from the properties of their time-1 maps. Therefore, we must build certain theoretical base and create certain new techniques to address issues raised in the case of flows.

We mention three differences of note between our setting and the existing literatures at both conceptual level and technical level which pervade the arguments in this paper: (1) At conceptual level, the most significant issue is that the {\em gap function} of a discrete time periodic orbit, i.e. the minimum separation of finite isolated points, is not well defined for continuous periodic orbit. Such a gap function plays a key role in the proof of \cite{HLMXZ}. (2) The presence of {\em shear}, i.e. the sliding of some orbits past other nearby orbits due to the slightly different speed at which they travel, is a typical phenomenon of continuous  time systems, which causes tremendous amount of "tail estimates" throughout this paper. (3)  Several main fundamental theoretical tools are not existing and need to be rebuilt from the base, such as  Anosov Closing Lemma,  Ma\~n\'e-Conze-Guivarc'h-Bousch's Lemma and  Periodic  Approximation Lemma.\\

\noindent{\bf Structure of this paper.} In Section \ref{S:Setting}, we set up the theoretic model and state the main results; In Section \ref{S:ProAxiomA}, we state (without proofs) some well known properties of Axiom A attractors, and some theoretical tools including  Anosov Closing Lemma,  Ma\~n\'e-Conze-Guivarc'h-Bousch's Lemma and  Periodic  Approximation Lemma preparing for the proof the main results; In Section \ref{S:PfMainThm}, we give the proof of Theorem \ref{T:MainResult}, of which proving Part I) of the Theorem costs most efforts; As follows, we leave the proofs of all the  technical lemmas to Section \ref{S:ProofsTechLem}. On one hand, readers may go through the main proof by assuming the validity of these technical lemmas without extra interruptions; on the other hand, these technical lemmas with their proofs may be of independent interest. Finally, we discuss the case when observables have higher regularity in Section \ref{S:HighReg} in which only some partial results are presented.

\section{main setting and results}\label{S:Setting}
Let $M$ be a compact smooth Riemannian manifold with Riemannian metric $d$ and $\Phi=\{\phi_t:M\to M\}_{t\in\mathbb{R}}$ be a $C^2$ flow on $M$.

\begin{defn}\label{D:AxiomAAttractor}
For $\L\subset M$, $(\L,\Phi)$ is called an Axiom A attractor if the following conditions hold:
\begin{itemize}
\item[A1)] $\L$ is a nonempty $\Phi$-invariant compact set.
\item[A2)] There exists an $\e_0>0$ such that for any $x\in M$ with $d(x,\L)<\e_0$
$$\lim_{t\to\infty}d\left(\phi_t(x),\L\right)=0.$$
\item[A3)] There exist $\l_0>0$, $C_0>1$ and a continuous splitting of tangential spaces  of $M$ restricted on $\L$, $T_xM=E^u_x\oplus E^c_x\oplus E^s_x\ \forall \ x\in \L$, such that the following hold
\begin{align*}
&(D_M\phi_t)_x(E^\tau_x)=E^\tau_{\phi_t(x)},\ \tau=u,c,s,\ \forall t\in\mb R\text{ and }x\in\L,\\
&\max\left\{\|(D_M\phi_{-t})_x|_{E^u_x}\|,\|(D_M\phi_{t})_x|_{E^s_x}\|\right\}\le C_0 e^{-t\l_0},\ \forall t\in \mb R^+,
\end{align*}
where $(D_M\phi_t)_x$ is the derivative of the time-$t$ map $\phi_t$ on $x$ with respect to space variables.
\item[A4)] $\inf_{x\in \L}\left\|\frac{d\phi_t(x)}{dt}\right\|>0\text{ and }E^c_x=span\left\{\frac{d}{dt}\phi_t(x)\right\},\ \forall x\in \L$.
\end{itemize}
\end{defn}

Denote by $\mathcal{M}(\L,\Phi)$ the set of all $\Phi$-invariant Borel probability measures on $\L$, which is a non-empty convex and compact topological space with respect to weak$^*$ topology. Denote by  $\mathcal{M}^{e}(\L,\Phi)\subset\mathcal{M}(\L,\Phi)$  the set of ergodic measures, which is the set of  the extremal points of $\mathcal{M}(\L,\Phi)$. Let $u:M\to\mathbb{R}$ and $\psi:M\to\mathbb{R}^{+}$ be continuous functions.  The quantity $\beta(u;\psi, \L, \Phi)$ being defined by
\begin{equation}\label{equ_ratiomimimalerg}
\beta(u;\psi,\L,\Phi):=\min_{\nu\in\mathcal{M}(\L,\Phi)}\frac{\int u d\nu}{\int \psi d\nu},
\end{equation}
is called the \emph{ratio minimum ergodic average}, and any $\nu\in\mathcal{M}(\L,\Phi)$ satisfying $$\frac{\int u d\nu}{\int \psi d\nu}=\beta(u;\psi,\L,\Phi)$$
is called a \emph{$(u,\psi)$-minimizing measure}. Denote that $$\mathcal{M}_{\min}(u;\psi,\L,\Phi):=\left\{\nu\in\mathcal{M}(\L,\Phi):\frac{\int ud\nu}{\int \psi d\nu}=\beta(u;\psi,\L,\Phi)\right\}.$$
By compactness of $\mathcal{M}(\L,\Phi)$, and the continuity of the operator $\frac{\int ud(\cdot)}{\int \psi d(\cdot)}$, it directly follows that $\mathcal{M}_{min}(u;\psi,\L,\Phi)\neq\emptyset$, which contains at least one ergodic $(u,\psi)$-minimizing measure by ergodic decomposition.

For $\a\in (0,1]$, let $\mathcal{C}^{0,\alpha}(M)$ be the space of $\alpha$-H\"{o}lder continuous real-valued functions on $M$ endowed with the $\alpha$-H\"{o}lder norm $\|u\|_{\alpha}:=\|u\|_{0}+[u]_{\alpha}$, where $\|u\|_{0}:=\sup_{x\in M}|u(x)|$ is the super norm, and $[u]_{\alpha}:=\sup_{x\neq y}\frac{|u(x)-u(y)|}{(d(x,y))^{\alpha}}$. Also note that when $\a=1$, $C^{0,1}(M)$ becomes the collection of all real-valued Lipschitz continuous functions, and  $[u]_{1}$ becomes the minimum Lipschitz constant of $u$. Additionally,  denote by $\mc C^{1,0}(M)$  the Banach space of continuous differentiable functions on $M$ endowed with the standard $C^1$-norm.

In this paper, we consider the weighted ergodic optimization problem and derive the following result.

\begin{thm}\label{T:MainResult}
Let $M$ be a compact smooth Riemannian manifold with Riemannian metric $d$ and $\Phi$ be a $C^2$ flow on $M$. Suppose that $(\L,\Phi)$ is an  Axiom A attractor, then the following hold:
\begin{itemize}
\item[I)] For $\a\in (0,1]$, given a $\psi\in \mc C^{0,\a}(M)$ with $\inf_{x\in M}\psi(x)>0$, then there exists an open and dense set $\mathfrak P\subset \mc C^{0,\a}(M)$ such that for any $u\in \mathfrak P$, the $(u|_\L,\psi|_\L)$-minimizing measure of $(\L,\Phi)$ is unique and is supported on a periodic orbit of $\Phi$.
\item[II)] For $\psi\in \mc C^{0,1}(M)$ with $\inf_{x\in M}\psi(x)>0$,  there exists an open and dense set $\mathfrak P\subset \mc C^{1,0}(M)$ such that for any $u\in \mathfrak P$, the $(u|_\L,\psi|_\L)$-minimizing measure of $(\L,\Phi)$ is unique and is supported on a periodic orbit of $\Phi$.
\end{itemize}
\end{thm}
{\bf We remark here that $M,\L,\Phi$ are assumed to satisfy conditions in Theorem \ref{T:MainResult} throughout the rest of this paper.}

\section{Properties of  Axiom A attractors}\label{S:ProAxiomA}
This section devotes to building theoretic tools as preparations for the proof of Theorem \ref{T:MainResult}.
\subsection{Invariant Manifolds}\label{S:InvMani}
For a point $x\in \L$ and $\epsilon>0$ the local stable and unstable sets are defined by
\begin{align*}
W_\epsilon^s(x)=&\{y\in M:d(\phi_t(x),\phi_t(y))\le\epsilon\ \forall t\ge0,\ d(\phi_t(x),\phi_t(y))\to0 \text{ as } t\to+\infty\},\\
W_\epsilon^u(x)=&\{y\in M:d(\phi_{-t}(x),\phi_{-t}(y))\le\epsilon\ \forall t\ge0,\ d(\phi_{-t}(x),\phi_{-t}(y))\to0 \text{ as } t\to+\infty\}.
\end{align*}
The following Lemma is a standard result of invariant manifolds in existing literature, of which the proof is omitted.
\begin{lem}\label{L:InvMani}
For any $\l_1\in (0,\l_0)$, there exists $\e_1>0$ and $C_1\ge 1$ such that for any $\e\in(0,\e_1]$, the following hold:
\begin{itemize}
\item[i)] $W^s_\e(x), W^u_\e(x)$ are $C^2$ embedded discs for all $x\in \L$ with $T_xW_\e^\tau(x)=E^\tau_{x}$, $\tau=u,s$;
\item[ii)] $d(\phi^t(x),\phi^t(y))\le C_1e^{-t\l_1}d(x,y)$ for $y\in W^s_\e(x)$, $t\ge 0$, and\\
$d(\phi^{-t}(x),\phi^{-t}(y))\le C_1e^{-t\l}d(x,y)$ for $y\in W^u_\e(x)$, $t\ge 0$;
\item[iii)] $W^s_\e(x), W^u_\e(x)$ vary continuously with respect to $x$ (in $C^1$ topology).
\end{itemize}
\end{lem}

By choosing the Riemannian metric, the Axiom A flow in Theorem \ref{T:MainResult} meets the following {\bf basic canonical setting} :  There are  positive constants $\delta, \epsilon,\beta,\lambda, C$ with $C\ge 1$ and  $\d\ll\e\ll\min\{\e_0,\e_1\}$, where $\e_0$ is as in A2) of the definition of Axiom A attractors and $\e_1$ is as in Lemma \ref{L:InvMani}, such that:

\begin{itemize}
	\item[(1)]
	For $x,y\in M$ with $d(x,y)\le\delta$, there is a unique time  $v=v(x,y)$ with $|v|\le Cd(x,y)$ such that
	\begin{itemize}
		\item[(a)] $W^s_\epsilon(\phi_{v}(x))\cap W^u_\epsilon(y)$ is not empty and contains only one element which is noted by  $w=w(x,y)$.
		\item[(b)] $d(x,y)\ge C^{-1}\max\{d(\phi_{v}(x),w),d(y,w),d(\phi_{v}(x),x), d(w,x)\}.$
	\end{itemize}
	\item[(2)] For $x\in M$ , $y\in W_\epsilon^u(x)$ and $t\ge 0$,  $d(\phi_{-t}x,\phi_{-t}y)\le Ce^{-\lambda t}d(x,y)$,\\ For $x\in M$ , $y\in W_\epsilon^s(x)$ and $t\ge 0$,  $d(\phi_tx,\phi_ty) \le Ce^{-\lambda t}d(x,y)$.
	\item[(3)] For $x,y\in M$, $d(\phi_tx,\phi_ty)\le Ce^{\beta |t|}d(x,y)$ for all $t\in\mathbb{R}$.
\end{itemize}
\begin{rem}\label{remark-1-1}
	In our following text, $\delta,\epsilon,\lambda,\beta,C$ are the positive constants as above. Additionally, for convenience, we assume $C\gg 1, 0<\delta\ll\epsilon\ll1$. Otherwise, we set a positive constant $\epsilon'$ such that $\epsilon'\ll\frac{\epsilon}{C^{10}e^{10\beta}}$. We set another positive constant $\delta'$ with $\delta'\ll\delta$ such that for  any $x,y\in M$ with $d(x,y)\le \frac{C^{10}e^{10\beta+10\lambda}}{e^{\lambda}-1}\delta'$, there is an unique time  $v=v(x,y)$ with $|v|\le Cd(x,y)$ such that $W^s_{\epsilon'}(\phi_{v}(x))\cap W^u_{\epsilon'}(y)$ is not empty and contains only one element.
\end{rem}

\begin{rem}\label{R:BasicCanSet}
For proofs and more details of Lemma \ref{L:InvMani} and the {\bf basic canonical setting}, we refer readers to \cite{PSh}, \cite{Bowen}, and \cite{Bowen-Ruelle}. The only property which is not appearing in the above references is the following inequality
\begin{equation}\label{E:ShearEst}
|v(x,y)|\le Cd(x,y)
\end{equation}
appearing  in (1) of {\bf basic canonical setting}. We remark here that this inequality holds when $\Phi$ is $C^2$. When $\Phi$ is $C^{1+\a}$ for some $\a\in(0,1]$, the above inequality will be replaced by $|v(x,y)|\le Cd^\a(x,y)$ which is still sufficient for the proof of this paper (although necessary modifications are required). This concludes that the Theorem \ref{T:MainResult} still holds for $C^{1+\a}$ flows.

Finally, to our knowledge, there is no explicit statement equivalent to (\ref{E:ShearEst}) in existing literature. Nevertheless, (\ref{E:ShearEst}) can be proved by combining Lemma 6, Proposition 8, Proposition 9 and Lemma 13 from \cite{Lian-Young}. Since (\ref{E:ShearEst}) is intuitively natural but at the same time the proof  involves considerable technical complexity, we decide not to put the detailed proof  in this paper for the sake of simplicity.
\end{rem}

\subsection{Anosov Closing Lemma}\label{S:AnosovClosingLem}
 Let $\delta',\epsilon',\delta,\epsilon$, $\lambda,\beta,C$ be the constants as in Remark \ref{remark-1-1}. Then we have the following Lemma.
\begin{lem}\label{An-1}Given $\eta\le \frac{C^{10}e^{10\beta+10\lambda}}{e^{\lambda}-1}\delta'$ and $T>0$,
	if  $x,y\in \L$ and continuous function $s:\mathbb{R}\to\mathbb{R}$  with $s(0)=0$ satisfy $$d(\phi_{t+s(t)}(y),\phi_t(x))\le \eta \text{ for } t\in [0,T],$$
	then  for all $t\in [0,T]$, the following hold:
	\begin{itemize}
	  \item[ASh1)] $|s(t)|\le 2C\eta$;
	  \item[ASh2)] $d(\phi_t\phi_{v(y,x)}(y),\phi_t(x))\le C^2e^{-\lambda\min(t,T-t)}\left(d(y,x)+d\left(\phi_{T}(y),\phi_T(x)\right)\right)$, where $v(y,x)$ is  as in Remark \ref{remark-1-1} satisfying $|v(y,x)|\le Cd(x,y)$.
\end{itemize}
	Especially, one has that
	\begin{itemize}
\item[(1)] If $d(\phi_{t+s(t)}(y),\phi_t(x))\le \eta$ for all $t\ge 0,$ then
		$$d(\phi_t\phi_{v(y,x)}(y),\phi_t(x))\to 0\text{ as } t\to +\infty.$$
		\item[(2)] If $d(\phi_{t+s(t)}(y),\phi_t(x))\le \eta$ for all $t\in\mathbb{R},$ then $\phi_{v(y,x)}(y)=x.$
	\end{itemize}
\end{lem}

  A segment of $\Phi$ is a curve  $\mathcal{S}: [a,b]\to M:t\to\phi_t(x)$ for some $x\in M$ and real numbers $a\le b$. We denote the left endpoint of $\mathcal{S}$ by  $\mathcal{S}^L=\phi_a(x)$,  the right endpoint of $\mathcal{S}$ by $\mathcal{S}^R=\phi_b(x)$ and the length of $\mathcal{S}$ by $|\mathcal{S}|=b-a$. By a segment $\mathcal{S}$, if $\mathcal{S}^L=\mathcal{S}^R$, we say $\mathcal{S}$ is a periodic segment. We have the following version of  Anosov Closing Lemma.
\begin{lem}[Anosov Closing Lemma]\label{lemma-a} 
 There are positive constants $L$ and $K$ depending on the system constants only such that if segment $\mathcal{S}$ of $\Phi|_\L$ satisfy
	\begin{itemize}
		\item[(a)]	$|\mathcal{S}|\ge K$;
		\item[(b)]$d(\mathcal{S}^L, \mathcal{S}^R)\le \delta'$.
	\end{itemize}
	Then, there is a periodic segment $\mathcal{O}$ such that $$\left| |\mathcal{S}|-|\mathcal{O}|\right|\le Ld(\mathcal{S}^L, \mathcal{S}^R)$$ and
	$$ d(\phi_t(\mathcal{O}^L),\phi_t(\mathcal{S}^L))\le Ld(\mathcal{S}^L, \mathcal{S}^R)\text{ for all } 0\le t\le \max(|\mathcal{S}|,|\mathcal{O}|).$$
\end{lem}
\begin{rem}
	In the following text, we also use $\mathcal{S}$ (so is $\mathcal{O}$ and $\mathcal{Q}$) to represnet the collection of points $\phi_t(\mathcal{S}^L), 0\le t\le |\mathcal{S}|$ as no confusion being caused. By Lemma \ref{An-1} and the choices of $\e$ and $\d$, $\mc O$ clearly belongs to $\L$.
\end{rem}
\subsection{Ma\~n\'e-Conze-Guivarc'h-Bousch's Lemma}
For $\g\in \mb R\setminus\{0\}$ and continuous function $u:M\to \mb R$, define that
\begin{equation}\label{E:u_g}
u_\g(x):=\frac1\g\int_0^\g u(\phi_t(x))dt.
\end{equation}

\begin{lem}[Ma\~n\'e-Conze-Guivarc'h-Bousch's Lemma]\label{MCGB}
For $0< \alpha\le 1$ and $N_0>0$, there exists a positive constant $\gamma=\gamma(\alpha)>N_0$ such that if $u\in\mathcal{C}^{0,\alpha}(M)$ satisfies $\beta(u;1,\L,\Phi|_\L)\ge0$, then there is a $v\in\mathcal{C}^{0,\alpha}(\L)$ such that $u_\gamma|_\L\ge v\circ \phi_\gamma|_\L-v.$
\end{lem}

\begin{rem}\label{rem-reveallossregularity}
We remark that the key point of Lemma \ref{MCGB} lies in the fact that $v$ is chosen with the same H\"{o}lder exponent as $u$. Indeed, there were a number of weak versions of Lemma \ref{MCGB} in the setting of smooth Anosov flows without fixed points, or certain expansive non-Anosov geodesic flows, where $v$ is still H\"{o}lder, but the H\"{o}lder exponent is less than $\alpha$ (for details, see \cite{Lopes-Rosas-Ruggiero,Lopes-Thieullen,Pollicott-Sharp}).
\end{rem}

By using Lemma \ref{MCGB}, we have the following Lemma.
\begin{lem}\label{revael}
For $0<\alpha\le 1$, there exists large $\gamma=\gamma(\alpha)$ such that, for $u\in\mathcal{C}^{0,\alpha}(M)$ and strictly positive $\psi\in\mathcal{C}^{0,\alpha}(M)$, there is a $v\in\mathcal{C}^{0,\alpha}(\L)$ such that
	\begin{itemize}
		\item[(1)] $u_\gamma|_\L-v\circ \phi_\gamma|_\L+v-\beta(u;\psi,\L,\Phi)\psi_\gamma|_\L\ge 0;$
		
		\item[(2)] $Z_{u,\psi}\subset\left\{x\in \L:\left(u_\gamma|_\L+v\circ \phi_\gamma|_\L-v-\beta(u;\psi,\L,\Phi)\psi_\gamma|_\L\right)(x)=0\right\},$
	\end{itemize}
where $Z_{u,\psi}=\overline{\cup_{\mu\in\mathcal{M}_{min}(u;\psi,\L,\Phi)}supp(\mu)}.$
\end{lem}

\begin{rem}\label{rem-reveal}For convenience, in the following text, if we need to use Lemma \ref{revael}, we use $\bar u$ to represent
	$u_\gamma|_\L+v\circ \phi_\gamma|_\L-v-\beta(u;\psi,\L,\Phi)\psi_\gamma|_\L$ for short. Then, $\bar u\ge 0$ and $Z_{u,\psi}\subset\{x\in{\L}:\bar u(x)=0\}.$
\end{rem}

\subsection{Periodic Approximation}\label{S:PeriodicApp}
For $\alpha\in(0,1]$, $Z\subset M$ and a segment $\mathcal{S}$ of $\Phi$, we define the {\it $\alpha$-deviation } of $\mathcal{S}$ with respect to $Z$ by
\[d_{\alpha,Z}(\mathcal{S})=\int_0^{|\mathcal{S}|}d^\alpha\left(\phi_t\left(\mathcal{S}^L\right),Z\right)dt.\]
For $P\ge 0$, using $\mathcal{O}^P_\L$ denote the collection of   periodic segments in $\L$ with length not larger than $P$. Now we have the following version of Quas and Bressaud's periodic approximation Lemma.
\begin{lem}\label{lemma-2}
 Let $Z\subset \L$ be a $\Phi$-invariant compact subset of $\L$.  Then, for all $\alpha\in(0,1],k\ge 0$,
	\[\lim_{P\to+\infty}P^k\min_{\mathcal{S}\in \mathcal{O}^P_\L}d_{\alpha,Z}(\mathcal{S})=0.\]
\end{lem}


\section{Proof of  Theorem \ref{T:MainResult}. }\label{S:PfMainThm}
This section contains two main subsections \ref{S:PfPartI} and \ref{S:ProofPartIIMainThm} corresponding to the proofs of Part I and Part II of Theorem \ref{T:MainResult} respectively. Indeed, Proposition \ref{prop-2} in Subsection \ref{S:MainProp} plays the key role in the proof of Theorem \ref{T:MainResult}, based on which the Part I result follows immediately and the Part II result follows in a straightforward way with the help of an approximation lemma. We also note that
throughout the whole section $\delta',\epsilon',\delta,\epsilon,\lambda,\beta,C$ are the fixed constants as in Remark \ref{remark-1-1}. 

\subsection{Proof of Part I) of Theorem \ref{T:MainResult}}\label{S:PfPartI}
This section mainly contains three parts: \ref{S:Locking}, \ref{S:GoodPerOrb} and \ref{S:MainProp}. As mentioned above, Proposition  \ref{prop-2} in Subsection \ref{S:MainProp} is the key to prove the main theorem of this paper, proving which is the main aim of  Section \ref{S:PfPartI}. While Section \ref{S:Locking} and \ref{S:GoodPerOrb} are devoted to building notions, tools and results as preparations for the proof of Proposition  \ref{prop-2}, of which Section \ref{S:Locking}  investigate the basic properties of  periodic orbits and Section \ref{S:GoodPerOrb} constructs periodic orbits with "good shapes".

\subsubsection{Locking Property of Periodic  Segments}\label{S:Locking}
 In this subsection, we show that periodic segments have locking property in some sense.

For $0<\eta\le\delta$, the {\it $\eta$-disk} of $x$ is defined by
\begin{equation}\label{E:StableCylinder}
\mc D(x,\eta)=\{y\in \L: d(x,y)\le\eta,W^s_\epsilon(x)\cap W^u_\epsilon(y)\neq\emptyset\}.
\end{equation}
 $\mc D(x,\eta)$ has the following properties:
\begin{itemize}
	\item[(a)] $W^s_\eta(x)\subset \mc D(x,\eta)$ and $\phi_{t}(W^s_\eta(x))\subset \mc D(\phi_{t}(x),Ce^{-\lambda t}\eta)$ for $t\ge 0$;
	\item[(b)] $W^u_\eta(x)\subset \mc D(x,\eta)$ and $\phi_{t}(W^u_\eta(x))\subset \mc D(\phi_{t}(x),Ce^{\lambda t}\eta)$ for $t\le 0$;
	\item[(c)]$\phi_t(\mc D(x,\eta))\subset \mc D\left(\phi_t(x),Ce^{\beta |t|}\eta\right)$ for $t\in\mathbb{R}$ satisfying $Ce^{\beta |t|}\eta<\delta$.
	\item[(d)] for $\eta\le\frac{\delta}{C}$ and $x,y\in \L$ with $d(x,y)\le \eta$, there exists a unique time  $v=v(x,y)$ with $|v|\le Cd(x,y)$ such that $y\in \mc D(\phi_v(x),\delta)$. In fact,  $v$ is the one given by the {\bf basic canonical setting}.
\end{itemize}
Now we define $D:\L\times \L\to[0,+\infty)$ by
\begin{equation}\label{E:DMetric}
D(x,y)=\left\{\begin{array}{ll}
\delta', & \text{ if } y\notin \mc D(x,\delta'),\\
d(x,y), & \text{ if } y\in \mc D(x,\delta').
\end{array}
\right.
\end{equation}
By a periodic segment $\mathcal{O}$ of $\Phi|_\L$, we define the {\it gap } of $\mathcal{O}$ by
\begin{equation}\label{E:GapFunc}
D(\mathcal{O})=\min_{x\in\mathcal{O},0<t<|\mathcal{O}|_{\min}}D(x,\phi_t(x)),
\end{equation}
where $|\mc O|_{\min}=\min\{t>0|\phi_t(x)=x,x\in \mc O\}$. In \cite{HLMXZ}, the gap function of  periodic orbits of a discrete time system plays an irreplaceable role in identifying the optimal periodic measures. Here, the gap of a periodic segment is an analogue of the gap function for discrete periodic orbits, which reflects the geometric characteristic of the periodic segment. Indeed, the definition of the gap of periodic segment is an conceptual enhancement, which originates the study of the ergodic optimization theory for the flow case.

Note that A4) of Definition \ref{D:AxiomAAttractor}  implies that $D(\mathcal{O})>0$ automatically, therefore in the rest of  this section we keep in mind that $D(\mathcal{O})>0$ without extra explanation.

Firstly, we present several technical Lemmas.

\begin{lem}\label{M-1}Let $\mathcal{O}$ be a periodic segment of $\Phi|_\L$. If $x,y\in \mathcal{O}$ satisfy $d(x,y)< \frac{D(\mathcal{O})}{C}$, then $\phi_v(x)=y$ where $v=v(x,y)$.
\end{lem}
\begin{proof} Let $\delta',\epsilon',\delta,\epsilon,\lambda,\beta,C$ be the constants as in Remark \ref{remark-1-1}.  Then,
	$$d(x,y)\le \frac{D(\mathcal{O})}{C}<\delta'\ll\delta.$$
	Hence, by the {\bf basic canonical setting}, there is a constant $v=v(x,y)$ such that
	$$y\in \mc D(\phi_v(x),Cd(x,y))\subset \mc D(\phi_v(x),\delta).$$
	If $\phi_v(x)\neq y$, then
	$$D(\mathcal{O})\le d(\phi_v(x),y)\le Cd(x,y)<D(\mathcal{O}),$$
	which is impossible. Thus, $\phi_v(x)= y$. This ends the proof.
\end{proof}
By a periodic segment $\mathcal{O}$ of $\Phi$,   the periodic measure $\mu_{\mathcal{O}}$ is defined by $$\mu_{\mathcal{O}}=\frac{1}{|\mathcal{O}|}\int_0^{|\mathcal{O}|}\delta_{\phi_t(\mathcal{O}^L)}dt.$$
By an ergodic measure $\mu\in\mathcal{M}^e(\L,\Phi|_\L)$, a point $x\in M$ is called a generic point of $\mu$ if the following holds
$$\lim_{T\to+\infty}\frac{1}{T}\int_0^Tf(\phi_t(x))dt=\int fd\mu\text{ for all } f\in \mathcal{C}(M).$$
The following Lemma shows that periodic segments have locking property in some sense.
\begin{lem}\label{M-3}
Let $\mathcal{O}$ be a periodic segment of $\Phi|_\L$ and $x\in M$. If
\begin{equation}\label{E:M-3-1}
d(\phi_t(x),\mathcal{O})\le\frac{D(\mathcal{O})}{4C^2e^\beta}\text{ for all }0\le t\le T,
\end{equation}
  then there is a $y\in\mathcal{O}$ such that
	 $$d(\phi_t(x),\phi_t(y))\le C d(\phi_t(x),\mathcal{O}) \text{ for all } 0\le t\le T.$$
Especially, if $d(\phi_t(x),\mathcal{O})\le\frac{D(\mathcal{O})}{4C^2e^\beta}$ for all $t\ge 0$,  then $x$ is a generic point of $\mu_{\mathcal{O}}.$
\end{lem}
\begin{proof} Let $\delta',\epsilon',\delta,\epsilon,\lambda,\beta,C$ be the constants as in Remark \ref{remark-1-1}. Take a positive constant  $\theta\ll \delta$ such that $d(\phi_t(z),z)\le\frac{D(\mathcal{O})}{4C}$ for all $|t|\le \theta$ and $z\in M$. By assumption (\ref{E:M-3-1}), there are $y_t'\in\mathcal{O}$ such that $$d(\phi_t(x),y_t')=d(\phi_t(x),\mathcal{O})\le\frac{D(\mathcal{O})}{4C^2e^\beta}\ll\delta\text{ for all } 0\le t\le T.$$
	Let $y_t=\phi_{v(y_t',\phi_t(x))}(y_t')$ for $0\le t\le T$. Then
	\begin{align}\label{em-1}
	\phi_{t}(x)\in \mc D\left(y_{t},Cd\left(\phi_t(x),\mathcal{O}\right)\right)\subset \mc D\left(y_{t},\frac{D(\mathcal{O})}{4Ce^\beta}\right)\text{ for all }0\le t\le T.
	\end{align}
	Fix $t_1,t_2\in [0,T]$ with $|t_1-t_2|\le\theta$. Then, by \eqref{em-1},
	\begin{align*}d(y_{t_1},y_{t_2})&\le d(y_{t_1},\phi_{t_1}(x))+d(\phi_{t_1}(x),\phi_{t_2}(x))+d(\phi_{t_2}(x),y_{t_2})\\
	&\le \frac{D(\mathcal{O})}{4Ce^\beta}+\frac{D(\mathcal{O})}{4C}+\frac{D(\mathcal{O})}{4Ce^\beta}<\frac{D(\mathcal{O})}{C}.
	\end{align*}
	Thus, by Lemma \ref{M-1}, there is a constant $\tau$ satisfying $|\tau|\le Cd(y_{t_1},y_{t_2})\ll\delta$ such that
	$$\phi_\tau(y_{t_1})=y_{t_2}.$$
	By the uniqueness of $v$ given in the {\bf basic canonical setting} and the smallness of both $\t$ and $|\tau|$, one has $\tau=t_2-t_1$. Hence,
	$$y_{t_2}=\phi_{t_2-t_1}(y_{t_1})\text{ for all }t_1,t_2\in [0,T]\text{ with } |t_1-t_2|\le\theta.$$
	Therefore, $y_t=\phi_t(y_0)$ for all $t\in[0,T]$ by induction, which implies that
	$$d(\phi_t(y_0),\phi_t(x))=d(y_t,\phi_t(x))\le Cd(y_t',\phi_t(x))=Cd(\phi_t(x),\mathcal{O})\text{ for all }0\le t\le T.$$
	Thus, $y=y_0$ is the point as required.

Now we assume that $d(\phi_t(x),\mathcal{O})\le\frac{D(\mathcal{O})}{4C^2e^\beta}$ for all $t\ge 0$. Then by the arguments above, there is $y\in \mc O$ such that $$d(\phi_t(x),\phi_t(y))\le \frac{D(\mathcal{O})}{4Ce^\beta}\le \delta' \text{ for all }t\ge 0.$$
Then by Lemma \ref{An-1}, we have
$$d(\phi_t\phi_v(x),\phi_t(y))\to 0\text{ as }t\to+\infty,$$
where $v=v(x,y)$.
Then $x$ must be a generic point of $\mu_\mathcal{O}$. This ends the proof.
\end{proof}

\begin{lem}\label{M-0}
There is positive constant $\tau_0$ such that for any $x\in \L$
$$\phi_t(x)\notin \mc D(x,\delta)\text{ for all }0<|t|\le\tau_0.$$
\end{lem}
\begin{proof} Let $\delta',\epsilon',\delta,\epsilon,\lambda,\beta,C$ be the constants as in Remark \ref{remark-1-1}. Take $\tau_0>0$  small enough such that $d(\phi_t(z),z)\le\frac{\delta'}{C^2}$ for all $|t|\le \tau_0$ and $z\in M$. Suppose that there is an $x\in \L$ and $0<|\tau|\le\tau_0$ such that $\phi_\tau(x)\in \mc D(x,\delta)$. Note $$w=W^s_\epsilon(x)\cap W^u_\epsilon(\phi_\tau(x)).$$
	Then, by (1)(b) of the {\bf basic canonical setting}, one has that
	$$\max\left\{d(w,x),d(w,\phi_\tau(x))\right\}\le Cd(x,\phi_\tau(x))\le \frac{\delta'}{C}.$$	
	Then for $t\ge 0$, $$d(\phi_t(w),\phi_t(x))\le Ce^{-\lambda t}d(w,x)\le Cd(w,x)\le \frac{\delta'} C,$$
	and for $t< 0$,
	$$d(\phi_t(w),\phi_t(x))\le d(\phi_t(w),\phi_t\phi_\tau(x))+d(\phi_\tau\phi_t(x),\phi_t(x))\le \frac{\delta'}{C}+\frac{\delta'}{C^2}\le \delta',$$
	where we used $w=W^s_\epsilon(x)\cap W^u_\epsilon(\phi_\tau(x))$ and the selection of $\tau$.
	Then by Lemma \ref{An-1}, there is a constant $l$ with $|l|\ll\delta$ such that $w=\phi_l(x)=\phi_{l-\tau}\phi_\tau(x)$. It is clear that at least one of $l$ and  $l-\tau$ is not zero since $\tau\neq 0$. Without loss of any generality, we assume that $l\neq 0$, then
	$\{x,\phi_l(x)\}\subset W^s_\epsilon(x)$ (otherwise $\{\phi_l(x),\phi_\tau(x)\}\subset W^u_\e(\phi_\tau(x))$).
	Thus
	$$W_\epsilon^s(x)\cap W^u_\epsilon(\phi_l(x))\neq\emptyset\text{ and }W_\epsilon^s(x)\cap W^u_\epsilon(x)\neq\emptyset,$$
	which is impossible by the uniqueness of function $v$ given in the {\bf basic canonical setting} and A4) of Definition \ref{D:AxiomAAttractor}. This ends the proof.
\end{proof}
\begin{rem}\label{rem-2-4}
	Lemma \ref{M-0}  provides a lower bound $\tau_0$ of the periods of periodic segments.
	\end{rem}
\begin{rem}
	We say a periodic segment $\mathcal{O}$ is {\it  pure}  if $\phi_t(y)\neq y$ for all $y\in\mathcal{O}$ and $0<t<|\mathcal{O}|$. By Lemma \ref{M-0}, a periodic segment $\mathcal{O}$ is pure if and only if $|\mc O|=|\mc O|_{\min}$.
\end{rem}

\subsubsection{Good periodic orbits}\label{S:GoodPerOrb}
In this subsection, we mainly demonstrate that for a given compact invariant set, there exist periodic segments being closed enough as well as with reasonable large gap, which are the candidates to support certain minimizing measures.
\begin{prop}\label{P:GoodPer}
For any $\a\in(0,1]$,  a  given $\tilde L>0$ and $\Phi$-forward-invariant non-empty subset $Z\subset \L$ (i.e. $\phi_t(Z)\subset Z,\ \forall t\ge 0$), there exists a periodic segment $\mc O$ of $\Phi_{\L}$ such that
\begin{align}\label{E:GapCond0}
	\frac{D^\alpha(\mathcal{O})}{d_{\alpha,Z}(\mathcal{O})}> \tilde L.
\end{align}
\end{prop}
\begin{proof}
Fix $0<\alpha\le 1$, and recall that $\delta',\epsilon',\delta,\epsilon,\lambda,\beta,C$ are as in Remark \ref{remark-1-1}, and $K$, $L$ are as in  Lemma \ref{lemma-a}. For the sake of convenience, we  assume that $K\ge L$ additionally. By Lemma \ref{lemma-2}, for any $k\in \mb N$, there exists a periodic segment $\mc O_0$ of $\Phi|_\L$ with period $P_0$ large enough such that
\begin{equation}\label{E:GP1}
d_{\a,Z}(\mc O_0)<P_0^{-k}<<\d'.
\end{equation}
We remark here that the period of a periodic segment is always assumed to be the {\bf MINIMUM} period, which will avoid unnecessary complexity without harming the argument.

If $D^\a(\mc O_0)>\tilde Ld_{\a,Z}(\mc O_0)$, the proof is done. Otherwise, one has that
\begin{equation}\label{E:GP2}
D^\a(\mc O_0)\le\tilde L d_{\a,Z}(\mc O_0)<\tilde LP_0^{-k}.
\end{equation}
Since $P_0,k$ can be chosen as large as needed, one can request that $D(\mc O_0)<\d'$. Therefore, by definition of $D(\mc O_0)$ (see (\ref{E:GapFunc})), there exist  $y\in \mc O_0$ and  $t_0\in(0,P_0)$ such that
$$\phi_{t_0}(y)\in \mc D(y,D(\mc O_0)).$$
Split the periodic segment $\mc O_0$ into two segments which are noted by
\begin{align*}
&\mc Q^1_0:[0,t_0]\to\L:t\to \phi_t(y);\\
&\mc Q^2_0:[t_0,P_0]\to\L:t\to\phi_t(y).
\end{align*}
We choose the segment with smaller length and note it by $\mc Q_0$. Then either $\mc Q_0^L\in \mc D(\mc Q^R,\d')$ or $\mc Q_0^R\in \mc D(\mc Q^L,\d')$. It is clear that in either case
$$d(\mc Q_0^L,\mc Q_0^R)\le \d'\text{ and }d_{\a,Z}(\mc Q_0)\le d_{\a,Z}(\mc O_0).$$

Next, we will estimate the increment  of orbit deviation after orbit splitting.
We will employ different discussions for  two different situations according to the length of the segment for which we set $3K$ as a landmark.

\noindent{\bf Case 1.} If the following condition holds
\begin{equation}\label{E:LargePerCond}
|\mc Q_0|>3K,
\end{equation}
also note that $d(\mc Q_0^L,\mc Q_0^R)\le \d'$, then Lemma \ref{lemma-a} is applicable here, by which one has that there exists a periodic segment $\mc O_1$ such that
\begin{equation}\label{E:O_1-1}
\left|\left|\mc Q_0\right|-\left|\mc O_1\right|\right|\le Ld(\mc Q_1^L,\mc Q_1^R)\le LD(\mc O_0)<L\tilde LP_0^{-k},
\end{equation}
\begin{equation}\label{E:O_1-2}
d\left(\phi_t\left(\mc Q_0^L\right),\phi_t\left(\mc O_1^L\right)\right)\le Ld(\mc Q_0^L,\mc Q_0^R)\ \forall 0<t<\max\{|\mc Q_0|,|\mc O_1|\}.
\end{equation}
Since $K\ge L$ and $\d'<<1$, (\ref{E:O_1-1}) together with the assumption $|\mc Q_0|>3K$ implies that
\begin{equation}\label{E:O_1-3}
|\mc O_1|\le \frac43 |\mc Q_0|\le \frac23|\mc O_0|.
\end{equation}

\begin{align*}\begin{split}
d_{\a,Z}(\mc O_1)=&\int_0^{|\mc O_1|}d^\a\left(\phi_t(\mc O_1^L),Z\right)dt \\
=&\int_0^{|\mc O_1|}d^\a\left(\phi_t\phi_{v(\mc O^L_1,\mc Q^L_0)}(\mc O_1^L),Z\right)dt \\
\le&\int_0^{|\mc Q_0|}d^\a\left(\phi_t\phi_{v(\mc O^L_1,\mc Q^L_0)}(\mc O_1^L),Z\right)dt \\
&+\int_{\left||\mc Q_0|-|\mc O_1|\right|}^{|\mc Q_0|+\left||\mc Q_0|-|\mc O_1|\right|}d^\a\left(\phi_t\phi_{v(\mc O^L_1,\mc Q^L_0)}(\mc O_1^L),Z\right)dt\\
\le&\int_0^{|\mc Q_0|}d^\a\left(\phi_t\phi_{v(\mc O^L_1,\mc Q^L_0)}(\mc O_1^L),\phi_t(\mc Q_0^L)\right)dt\quad\quad \quad\quad\quad{\bf \left(Int(a)\right)}\\
&+\int_0^{|\mc Q_0|}d^\a\left(\phi_t(\mc Q^L_0),Z\right)dt\quad\quad \quad\quad\quad\quad\quad\quad\quad\ \quad{\bf \left(Int(b)\right)}\\
&+\int_0^{|\mc Q_0|}d^\a\left(\phi_{\left||\mc Q_0|-|\mc O_1|\right|}\phi_t\phi_{v(\mc O^L_1,\mc Q^L_0)}(\mc O_1^L),Z\right)dt\quad\quad{\bf \left(Int(c)\right)}.
\end{split}
\end{align*}
By applying Ash2) of Lemma \ref{An-1}, one has that
\begin{equation}\label{E:Int(a)}
{\bf Int(a)}\le \int_0^{|\mc Q_0|}2C^2e^{-\l\min(t,|\mc Q_0|-t)}Ld^\a\left(\mc Q_0^L,\mc Q_0^R\right) dt\le \frac{2(2C^2L)^\a}{\l\a}d^\a(\mc Q_0^L,\mc Q_0^R).
\end{equation}
By definition, one has that
\begin{equation}\label{E:Int(b)}
{\bf Int(b)}=d_{\a,Z}(\mc Q_0).
\end{equation}
By (3) of the {\bf basic canonical setting}, one has that
\begin{align}\begin{split}\label{E:Int(c)}
{\bf Int(c)}= &\int_0^{|\mc Q_0|}d^\a\left(\phi_{\left||\mc Q_0|-|\mc O_1|\right|}\phi_t\phi_{v(\mc O^L_1,\mc Q^L_0)}(\mc O_1^L),\phi_{\left||\mc Q_0|-|\mc O_1|\right|}Z\right)dt\\
\le&\left(Ce^{\b\left||\mc Q_0|-|\mc O_1|\right|}\right)^\a\int_0^{|\mc Q_0|}d^\a\left(\phi_t\phi_{v(\mc O^L_1,\mc Q^L_0)}(\mc O_1^L),Z\right)dt\\
\le&\left(Ce^{\b\left||\mc Q_0|-|\mc O_1|\right|}\right)^\a\left({\bf Int(a)+Int(b)}\right).
\end{split}
\end{align}
By taking $P_0$ and $k$ large, one is able to make $\left||\mc Q_0|-|\mc O_1|\right|<1$. Therefore, one has the following simplified estimate
\begin{equation}\label{E:DaZO1}
d_{\a,Z}(\mc O_1)\le L_1 d^\a(\mc Q_0^L,\mc Q_0^R)+L_2 d_{\a,Z}(\mc Q_0)\le \hat L d_{\a,Z}(\mc O_0),
\end{equation}
where
\begin{align*}
L_1&=\left(C^\a e^{\a\b}+1\right)\frac{2(2C^2L)^\a}{\l\a};\\
L_2&=C^\a e^{\a\b}+1;\\
\hat L&=L_1\tilde L+L_2=\left(C^\a e^{\a\b}+1\right)\left(\frac{2(2C^2L)^\a}{\l\a}\tilde L+1\right).
\end{align*}

Once $d_{\a,Z}(\mc O_1)>\tilde L d_{\a,Z}(\mc O_1)$, $\mc O_1$ is the periodic segment as required, the splitting process stops. Otherwise, repeat the operation above as long as it is doable. Note that such a process will stop at a finite time, since the operation above will reduce the period of periodic segment at least $\frac13$ by (\ref{E:O_1-3}). Therefore, under the assumption that the operation above  is always doable, the process will end on an periodic segment $\mc O_m$ for some $m\in\mb N\cup\{0\}$, which either satisfies the requirement of this Proposition or $|\mc O_m|\ge 3K$ while $|\mc Q_m|<3K$. In either cases, one has that
$$m\le \frac{\log P_0-\log(3K)}{\log 1.5}+1,$$
and
$$d_{\a,Z}(\mc O_i)\le \hat L^id_{\a,Z}(\mc O_0),\ \forall 1\le i\le m.$$
In order to make each operation above doable, one needs that
$$D(\mc O_i)<\d',\ \forall 1\le i\le m-1,$$
which can be done by  assuming the largeness of $P_0$ and $k$ in advance. To be precise, one can take
\begin{equation}\label{E:CondPk1}
k>\frac{\log \hat L}{\log 1.5}\text{ and }P_0^{\frac{\log \hat L}{\log 1.5}-k}<\frac{(\d')^\a}{\tilde L\hat L},
\end{equation}
 where the second inequality above  implies that for all $0\le i\le m-1$
 $$\tilde Ld_{\a,Z}(\mc O_i)\le P_0^{-k}\tilde L\hat L^{m}<(\d')^\a,$$
 which ensures the existence of $\mc O_{i+1}$ and $ \tilde L d_{\a,Z}(\mc O_m)<(\d')^\a$.

\noindent{\bf Case 2.} We will deal with the case that $|\mc Q_m|\le 3K$, which is the counterpart of the case when (\ref{E:LargePerCond}) holds.  We will show that by rearranging extra largeness of $P_0$ and $k$, one can make $\mc O_m$ satisfy the requirement of Proposition \ref{P:GoodPer}. We will prove this by contradiction.

Before going to further discussion, we should note first  that the union of all periodic orbits of $\Phi|_\L$ with period $\le 4K$ is a nonempty compact  subset of $\L$, which is denoted by $Per_{4K}$. Once $Z\cap Per_{4K}\neq \emptyset$ Proposition \ref{P:GoodPer} holds automatically; otherwise, there exists $\s>0$ such that
 \begin{equation}\label{E:4KSeparation}
 d(x,Z)>\s\ \forall x\in Per_{4K}.
 \end{equation}

  Suppose that
  \begin{equation}\label{E:FakeCond}
  D^\a(\mc O_m)\le \tilde L d_{\a,Z}(\mc O_m)<(\d')^\a.
  \end{equation}
  When
  \begin{equation}\label{E:Kto3K}
  K\le |\mc Q_m|\le3K,
  \end{equation}
   by the exactly same argument as on $\mc Q_0$, one has that there exists a periodic segment $\mc O_{m+1}$ such that
  $$|\mc O_{m+1}|\le 4K\text{ and }d_{\a,Z}(\mc O_{m+1})\le \hat L d_{\a,Z}(\mc O_m)\le \hat L^{m+1}d_{\a,Z}(\mc O_0)<\hat L^{m+1}P_0^{-k}.$$
  By choosing $P_0$ and $k$ large enough one can make $d_{\a,Z}(\mc O_{m+1})<\s$ which implies an contradiction with (\ref{E:4KSeparation}). Therefore (\ref{E:FakeCond}) and (\ref{E:Kto3K}) can not hold simultaneously for  $P_0$ and $k$ large enough.

 When
 \begin{equation}\label{E:0toK}
 |\mc Q_m|< K,
 \end{equation}
  Lemma \ref{lemma-a} is not applicable directly. For sake of convenience, note $l=|\mc Q_m|$. By (\ref{E:FakeCond}),  $\mathcal{Q}_m^L\in D(\mathcal{Q}_m^R,\delta')$ or $\mathcal{Q}_m^R\in D(\mathcal{Q}_m^L,\delta')$. Then by Lemma \ref{M-0}, $l>\tau_0$. Let $q$ be the integer such that
$$K\le ql\le 2K\text{ and thus } 2\le q\le \left[\frac{2K}{\tau_0}\right].$$ Note
$$\mathcal{S}_i:[0,l]\to M: t\to \phi_{il+t}(\mathcal{Q}_m^L)\text{ for }i=0,1,2,\cdots, q-1,$$
and
$$\mathcal{S}:[0,ql]\to M: t\to \phi_{t}(\mathcal{Q}_m^L).$$
Then,
$$d(\mathcal{S}_0^L,\mathcal{S}_0^R)=d(\mathcal{Q}_m^L,\mathcal{Q}_m^R)$$
$$d(\mathcal{S}_1^L,\mathcal{S}_1^R)=d(\phi_l(\mathcal{S}_0^L),\phi_l(\mathcal{S}_0^R))\le Ce^{\beta l} d(\mathcal{S}_0^L,\mathcal{S}_0^R)\le Ce^{\beta K} d(\mathcal{Q}_m^L,\mathcal{Q}_m^R) ,$$
$$\dots$$
$$d(\mathcal{S}_{q-1}^L,\mathcal{S}_{q-1}^R)\le (Ce^{\beta K} )^{q-1}d(\mathcal{Q}_m^L,\mathcal{Q}_m^R).$$
Therefore,
\begin{equation*}
d(\mathcal{S}^L,\mathcal{S}^R)\le \sum_{i=0}^{q-1}(Ce^{\beta K} )^{i}d(\mathcal{Q}_m^L,\mathcal{Q}_m^R)\le \frac{(Ce^{\beta K} )^{\left[\frac{2K}{\tau_0}\right]}-1}{Ce^{\beta K} -1} D(\mc O_m),
\end{equation*}
which together with (\ref{E:FakeCond}) implies that
\begin{equation}\label{E:SGap}
d(\mathcal{S}^L,\mathcal{S}^R)\le \frac{(Ce^{\beta K} )^{\left[\frac{2K}{\tau_0}\right]}-1}{Ce^{\beta K} -1} \left(\tilde Ld_{\a,Z}(\mc O_m)\right)^{\frac1\a}\le \frac{(Ce^{\beta K} )^{\left[\frac{2K}{\tau_0}\right]}-1}{Ce^{\beta K} -1} \left(\tilde L\hat L^mP_0^{-k}\right)^{\frac1\a}.
\end{equation}
By taking $P_0$ and $k$ large enough, one can make $d(\mc S^L,\mc S^R)<\d'$. Also note that $|\mc S|\ge K$, then Lemma \ref{lemma-a} is applicable to $\mc S$. Therefore, there exists a periodic segment $\mathcal{O}_*$ such that $|\mathcal{O}_*|\le |\mathcal{S}|+Ld(\mathcal{S}^L,\mathcal{S}^R)\le 3K$ and
\begin{align}\label{E-11}
d\left(\phi_t\phi_{v(\mathcal{O}_*^L,\mathcal{S}^L)}(\mathcal{O}_*^L),\phi_t(\mathcal{S}^L)\right)\le L\frac{(Ce^{\beta K} )^{\left[\frac{2K}{\tau_0}\right]}-1}{Ce^{\beta K} -1} \left(\tilde L\hat L^mP_0^{-k}\right)^{\frac1\a}\ \forall 0\le t\le |\mathcal{S}|,
\end{align}
where the right hand side of the above inequality can be make smaller than $\frac13\s$ by taking $P_0$ and $k$ large enough.
On the other hand,
$$d_{\alpha,Z}(\mathcal{S}_0)=d_{\alpha,Z}(\mathcal{Q}_m)$$
$$d_{\alpha,Z}(\mathcal{S}_1)=\int_0^ld^\alpha(\phi_{l+t}(\mathcal{S}_0^L),Z)\le (Ce^{\beta l})^\alpha d_{\alpha,Z}(\mathcal{S}_0)\le (Ce^{\beta K})^\alpha d_{\alpha,Z}(\mathcal{Q}_m),$$
$$\dots$$
$$d_{\alpha,Z}(\mathcal{S}_{q-1})\le(Ce^{\beta K})^{(q-1)\alpha} d_{\alpha,Z}(\mathcal{Q}_m).$$
Thus,
\begin{align}
\begin{split}\label{E:SSeparation}
d_{\alpha,Z}(\mathcal{S})&=\sum_{i=0}^{q-1}d_{\alpha,Z}(\mathcal{S}_{i})\le \sum_{i=0}^{q-1}(Ce^{\beta K})^{i\alpha} d_{\alpha,Z}(\mathcal{Q}_m)\\
&\le\sum_{i=0}^{[\frac{2K}{\tau_0}]-1}C^{i\a}e^{i\beta K\a}d_{\alpha,Z}(\mathcal{O}_{m})\\
&\le \sum_{i=0}^{[\frac{2K}{\tau_0}]-1}C^{i\a}e^{i\beta K\a} \hat L^{m}d_{\alpha,Z_{u}}(\mathcal{O}_0)\\
&\le \frac{(Ce^{\beta K} )^{\left[\frac{2K}{\tau_0}\right]\a}-1}{\left(Ce^{\beta K}\right)^\a -1} \hat L^mP_0^{-k},
\end{split}
\end{align}
which can be make smaller than $\frac13\s$ by taking  $P_0$ and $k$ large enough.
Since $|\mathcal{S}|\ge K>1$, there is a point $t_*\in[0,|\mathcal{S}|]$ such that $$d(\phi_{t_*}(\mathcal{S}^L),Z)\le \frac{\sigma}{3}.$$
Therefore, by \eqref{E-11}, by taking $P_0$ and $k$ large enough, we have
$$d(\phi_{t_*}\phi_{v_*}(\mathcal{O}_*^L),Z)\le d\left(\phi_{t_*}\phi_{v(\mathcal{O}_*^L,\mathcal{S}^L)}(\mathcal{O}_*^L),\phi_{t_*}(\mathcal{S}^L))+d(\phi_{t_*}(\mathcal{S}^L),Z\right)\le \frac{2}{3}\s<\sigma$$
which contradicts with (\ref{E:4KSeparation}) as $K\le |\mc S|\le 2K$.\\
Hence  (\ref{E:FakeCond}) can not hold for large enough $P_0$ and $k$. This ends the proof.
\end{proof}

Here, we remark that there is no fixed point under the setting of this paper by A4) of Definition \ref{D:AxiomAAttractor}, thus the orbit splitting process cannot stop at a fixed point, which is the main difference comparing to the discrete time case in \cite{HLMXZ}. The Case 2. above is mainly taking care this issue.

\subsubsection{Main Proposition}\label{S:MainProp}
In this subsection, we state and prove our main proposition.
For a continuous function $u$ and a segment $\mc S$ of $\Phi$, define the integration of $u$ along $\mc S$ with time interval $[a,b]$ and starting point $x$ by the following
\begin{equation}\label{E:IntSeg}
\langle \mc S,u\rangle:=\int_a^bu\left(\phi_t\left(\mc S^L\right)\right)dt,
\end{equation}
also recall that the definition of $u_\g$ for $\g>0$ by
$$u_\g(x)=\frac1{\g}\int_0^\g u\left(\phi_t(x)\right)dt.$$
Now we have the following proposition.
\begin{prop}\label{prop-2}
	Given $0<\varepsilon\le 1$, $0< \alpha\le 1$, a strictly positive function $\psi\in\mathcal{C}^{0,\alpha}(M)$ and $u\in \mathcal{C}^{0,\alpha}(M)$, if a periodic segment $\mathcal{O}$ of $\Phi|_\L$ satisfies the following comparison condition
		\begin{align}\label{condition}
		\frac{D^\alpha(\mathcal{O})}{d_{\alpha,Z_{u,\psi}}(\mathcal{O})}>
		\left(\frac{4C^{3\alpha}(\|\bar u\|_\alpha+10\varepsilon+\frac{ \|\bar u\|_0 + 1 }{\psi_{min} }\|\psi_\gamma\|_\alpha) }{\lambda\alpha\varepsilon \tau_0}+1+\frac{1}{\tau_0}\right) \frac{\|\bar u\|_\alpha\|\psi\|_0}{\psi_{min}}\cdot\frac{ 100(4C^3e^{2\beta})^\alpha}{\varepsilon},
		\end{align}
	where  $\bar u$ is defined in Remark \ref{revael} and $\tau_0$ is the constant in Lemma \ref{M-0},
	then the periodic measure  $$\mu_{\mathcal{O}}\in\mathcal{M}_{min}(u+\varepsilon d^\alpha(\cdot,\mathcal{O}) +h;\psi,\L,\Phi),$$ where $h\in\mc C^{0,\a}(M)$ satisfying $\|h\|_\alpha<10\varepsilon$ and
	$$\|h\|_0<\min\left\{\frac{\frac{\varepsilon}{2}\left(\frac{D(\mathcal{O})}{4C^3e^{2\beta}}\right)^\alpha}{\left(\frac{4C^{3\alpha}(\|\bar u\|_\alpha+10\varepsilon+\frac{ \|\bar u\|_0 + 1 }{\psi_{min} }\|\psi_\gamma\|_\alpha) }{\lambda\alpha\varepsilon}+|\mathcal{O}|+1\right) \frac{\|\psi\|_0+\psi_{min}}{\psi_{min}}},1\right\}.$$
\end{prop}
\begin{proof}Fix $\varepsilon, \alpha,\mathcal{O},\psi,u,h$ as in the Proposition, $\delta',\epsilon',\delta,\epsilon,\lambda,\beta,C$ as in Remark \ref{remark-1-1}, $\bar u,Z_{u,\psi},\gamma,\psi_\gamma$ as in Remark \ref{rem-reveal}, $\tau_0$ as in Lemma \ref{M-0}. 	Note $$G=\bar u+\varepsilon d^\alpha(\cdot,\mathcal{O})+h-a_\mathcal{O}\psi_\gamma$$ where
	\begin{align*}
	\begin{split}
	a_{\mathcal{O}}&:=\frac{\left\langle \mathcal{O},\bar u+\varepsilon d^\alpha(\cdot,\mathcal{O})+h\right\rangle }{\left\langle \mathcal{O},\psi_\gamma\right\rangle}=\frac{\left\langle \mathcal{O},\bar u+h\right\rangle }{\left\langle \mathcal{O},\psi_\gamma\right\rangle}.
	\end{split}\end{align*}
	By straightforward computation, one has that
	\begin{align}\label{X-1}
	\begin{split}
	\left| a_{\mc O}\right|\le\frac{\left\langle \mathcal{O},\|\bar u\|_\alpha d^\alpha(\cdot,Z_{u,\psi})+h\right\rangle}{\left\langle \mathcal{O},\psi_{min}\right\rangle}\le \frac{\|\bar u\|_\alpha d_{\alpha,Z_{u,\psi}}(\mathcal{O})}{| \mathcal{O}|\psi_{min}}+\frac{\|h\|_0}{\psi_{min}},
	\end{split}
	\end{align}
	where we used $\bar u|_{Z_{u,\psi}}=0$.
	Notice that for all $\mu\in\mathcal{M}(\L,\Phi)$
	\begin{align*}\frac{\int\left( u+\varepsilon d^\alpha(\cdot,\mathcal{O}) +hd\right)\mu}{\int \psi d\mu}&=\frac{\int \left(u_\gamma+\varepsilon d^\alpha(\cdot,\mathcal{O}) +h\right)d\mu}{\int \psi_\gamma d\mu}\\
	&=\frac{\int \left(\bar u+\varepsilon d^\alpha(\cdot,\mathcal{O}) +h\right)d\mu}{\int \psi_\gamma d\mu}+\beta(u;\psi,\L,\Phi)\\
	&=\frac{\int Gd\mu}{\int \psi d\mu}+a_{\mathcal{O}}+\beta(u;\psi,\L,\Phi).
	\end{align*}
	Then, in order to show that $\mu_{\mathcal{O}}\in\mathcal{M}_{min}(u+\varepsilon d^\alpha(\cdot,\mathcal{O})+h;\psi,\L,\Phi)$,
	it is enough to show that $\mu_{\mathcal{O}}\in\mathcal{M}_{min}(G;\psi_\gamma,\L,\Phi)$. Since $\psi$ is strictly positive and $\int Gd\mu_\mathcal{O}=0$, it is enough to show that
	\begin{equation}\label{E:PositiveIntG}
	\int Gd\mu\ge0\text{ for all }\mu\in\mathcal{M}^e(\L,\Phi).
	\end{equation}
Define a compact set $\mc R\subset M$ by
	\[\mc R=\left\{y\in M:d(y,\mathcal{O})\le\left(\frac{|a_{\mathcal{O}}|\|\psi\|_0+\|h\|_0}{\varepsilon}\right)^{\frac{1}{\alpha}}\right\}.\]
	We have the following Claim.
	
	\noindent{\it {\bf Claim 1.} $\mc R$ contains all $x\in M$ with $G(x)\le0$. }
	\begin{proof}[Proof of Claim 1.]
		Given $x\in M\setminus \mc R$, we are to show that $G(x)>0$. Note that
		\begin{align}\label{X-2}
		\begin{split}\bar u+h-a_{\mathcal{O}}\psi_\gamma&\ge-|a_{\mathcal{O}}|\|\psi_\gamma\|_0-\|h\|_0.
		\end{split}
		\end{align}
	where we used $\bar u\ge0$ and $\|\psi_\gamma\|_0\le \|\psi\|_0$. Then
		\begin{align*}G(x)&=\bar u(x)+\varepsilon d^\alpha(x,\mathcal{O})+h(x)-a_{\mathcal{O}}\psi_\gamma\\
		&\ge \varepsilon d^\alpha(x,\mathcal{O})-|a_{\mathcal{O}}|\|\psi\|_0-\|h\|_0\\
		&> \varepsilon\cdot\left(\left(\frac{|a_{\mathcal{O}}|\|\psi\|_0+\|h\|_0}{\varepsilon}\right)^{\frac{1}{\alpha}}\right)^\alpha-|a_{\mathcal{O}}|\|\psi\|_0-\|h\|_0\\
		&=0.	
		\end{align*}
	   This ends the proof of Claim 1.
   \end{proof}
Define a compact set $\mc R'\subset M$ by
	\[\mc R'=\left\{y\in M:d(y,\mathcal{O})\le\left(\frac{2(|a_{\mathcal{O}}|\|\psi\|_0+\|h\|_0)}{\varepsilon}\right)^{\frac{1}{\alpha}}\right\}.\]
	It is easy to see that $\mc R$ is in the interior of $\mc R'$ and the following holds because of (\ref{condition}), (\ref{X-1}) and the range of $\|h\|_0$
	\begin{equation}\label{E:R'Est}
	d(y,\mc O)\le\left(\frac{2(a_{\mathcal{O}}\|\psi\|_0+\|h\|_0)}{\varepsilon}\right)^{\frac{1}{\alpha}}\le\frac{D(\mathcal{O})}{4C^3e^{2\beta}},\ \forall y\in \mc R'.
	\end{equation}
	 By Claim 1, there is a constant $\tau$ with $0<\tau<1$ such that $G(\phi_t(x))>0$ for all $x\in M\setminus \mc R'$ and $|t|\le\tau$. 
	 
	 Now we claim the following assertion:\\
	{\it {\bf Claim 2.} If $z\in M$ is not a generic point of $\mu_\mathcal{O}$, then there is $m\ge\tau$ such that
	$$\int_{0}^mG(\phi_t(z))dt>0.$$}
		
Next we prove the Proposition by assuming the validity of Claim 2, while the proof of Claim 2. is left to the end of this section. For a given ergodic measure $\mu\in\mathcal{M}^e(\L,\Phi)$, if $\mu=\mu_{\mathcal{O}}$, (\ref{E:PositiveIntG}) obviously holds.  Otherwise, let $z$ be a generic point of $\mu$, thus $z$ is not a generic point of $\mu_{\mathcal{O}}$. Therefore, by Claim 2, there is a $t_1\ge\tau$ such that \[\int_0^{t_1} G(\phi_t(z))dt>0.\]
	Since $\phi_{t_1}(z)$ is still not a generic point of $\mu_{\mathcal{O}}$, by using claim 2 agian, one has a $t_2\ge t_1+\tau$ such that \[\int_{t_1}^{t_2} G(\phi_t(z))dt>0.\]
	By repeating the above process, one has $0\le t_1<t_2<t_3<\cdots$ with all gaps not less than $\tau$ such that
	\[\int_{t_i}^{t_{i+1}} G(\phi_t(z))dt>0\text{ for } i=0,1,2,3,\cdots,\]
	where we assign $t_0=0$. Therefore
	\begin{align*}
	\int Gd\mu&=\lim_{l\to+\infty}\frac{1}{l}\int_{0}^{l}G(\phi_t(z))dt\\
	&= \lim_{i\to+\infty}\frac{1}{t_i}\left(\int_{t_0}^{t_1}G(\phi_t(z))dt+\int_{t_1}^{t_2}G(\phi_t(z))dt+\cdots+\int_{t_{i-1}}^{t_i}G(\phi_t(z))dt\right)\\
	&\ge 0.
	\end{align*}
	Thus, $\mu_{\mathcal{O}}\in\mathcal{M}_{min}(u+\varepsilon d^\alpha(\cdot,\mathcal{O})+h;\psi,\L,\Phi)$. This ends the proof.
	\end{proof}
\begin{rem}\label{remark} It is not difficult to see that for any $\varepsilon'>\varepsilon$, $\mu_\mathcal{O}$ is the unique measure in $\mathcal{M}_{min}(u+\varepsilon' d^\alpha(\cdot,\mathcal{O})+h;\psi,\L,\Phi)$  whenever $\|h\|_\alpha<10\varepsilon$ and
	$\|h\|_0$ is sufficiently small. The Proposition shows that there is an open set of $\mathcal{C}^{0,\alpha}(M)$ near $u$ such that these $\alpha$-H\"older functions in the open set has the same unique minimizing measure  with respect to $\psi$ being supported on a periodic orbit.
\end{rem}
	\begin{proof}[Proof of Claim 2.]
			If $z\notin \mc R'$, just take $m=\tau$,
		we have nothing to prove  since $G(\phi_t(z))>0$ for all $|t|\le\tau$.	
		Therefore, one needs only to consider the case that $z\in \mc R'$. Also note that, since $z$ is not a generic point of $\mu_\mathcal{O}$, Lemma \ref{M-3} implies that the following inequality
		$$d(\phi_t (z),\mathcal{O})\le\frac{D(\mathcal{O})}{4C^2e^\beta}$$
		{\bf CANNOT} hold for all $t\ge 0$. Thus there is an $m_1>0$ such that $$d(\phi_{m_1} (z),\mathcal{O})>\frac{D(\mathcal{O})}{4C^2e^\beta}.$$
		Let $m_2>0$ be the smallest time such that
		\begin{align}\label{X-3}
		d(\phi_{m_2} (z),\mathcal{O})=\frac{D(\mathcal{O})}{4C^2e^\beta},
		\end{align}
		where the existence of such $m_2$ is ensured by (\ref{E:R'Est}) and the continuity of the flow.
Then, by (3) of the {\bf basic canonical setting}, one has the following
 $$d(\phi_{m_2-t}(z),\mathcal{O})>\frac{D(\mathcal{O})}{4C^3e^{2\beta}},\ \forall 0< t\le 1,$$
 which together with (\ref{E:R'Est}) implies that
 \begin{align}\label{Y-2}
	\phi_{m_2-t}(z)\notin \mc R'\text{ for all }0< t\le 1.
		\end{align} 	
Thus
	\begin{align}\label{X-5}
		\begin{split}\int_{m_2-1}^{m_2}G(\phi_{t}(z))dt&=\int_{m_2-1}^{m_2} \bar{u}(\phi_{t}(z))+\varepsilon d^\alpha(\phi_{t}(z),\mathcal{O})+h(\phi_{t}(z))-a_{\mathcal{O}}\psi_\gamma(\phi_{t}(z))dt\\
		&\ge \int_{m_2-1}^{m_2}\varepsilon d^\alpha(\phi_{t}(z),\mathcal{O})-|a_{\mathcal{O}}|\|\psi\|_0-\|h\|_0dt\\
		&\ge\varepsilon\cdot\left(\frac{D(\mathcal{O})}{4C^3e^{2\beta}}\right)^\alpha-|a_{\mathcal{O}}|\|\psi\|_0-\|h\|_0,\\
		\end{split}
		\end{align}
		where we used \eqref{X-2}.

	Now since $\mc R'$ is compact, there is an $m_3$ which is the largest time such that $0\le m_3\le m_2$
		and $\phi_{t}(z)\in \mc R'$.
	By \eqref{Y-2}, it is clear that $m_3\le m_2-1$.
		Then by Claim 1,  for all $m_3<t<m_2-1$ \begin{align}\label{X-12}
		G(\phi_t(z))>0,\end{align}
		where we used the fact $\mc R\subset \mc R'$.
On the other hand, since  $m_3<m_2$,  by the choice of $m_2$ (see \eqref{X-3}), one has that
		$$d(\phi_t(z),\mathcal{O})\le \left(\frac{2(a_{\mathcal{O}}\|\psi\|_0+\|h\|_0)}{\varepsilon}\right)^{\frac{1}{\alpha}}<\frac{D(\mathcal{O})}{4C^2e^\beta}\text{ for all } 0\le t\le m_3.$$ Therefore, by Lemma \ref{M-3}, there is $y_0\in\mathcal{O}$ such that
		$$d(\phi_t(z),\phi_t(y_0))\le C\left(\frac{2(a_{\mathcal{O}}\|\psi\|_0+\|h\|_0)}{\varepsilon}\right)^{\frac{1}{\alpha}}\le \delta'\text{ for all }t\in[0,m_3].$$
	By using Lemma \ref{An-1}, we have for all $0\le t\le m_3$,
		\begin{align*}
d(\phi_{t}\phi_{v(y_0,z)}(y_0),\phi_t(z))\le 2C^2e^{-\lambda\min(t,m_3-t)}C\left(\frac{2(a_{\mathcal{O}}\|\psi\|_0+\|h\|_0)}{\varepsilon}\right)^{\frac{1}{\alpha}}.
		\end{align*}
Hence,
		\begin{align*}
		\begin{split}
		&\int_{0}^{ m_3}d^\alpha(\phi_{t}(z),\phi_{t}\phi_{v(y_0,z)}(y_0))dt\\
		\le&\int_{0}^{ m_3}\left(2C^3\left(\frac{2\left(|a_{\mathcal{O}}|\|\psi\|_0+\|h\|_0\right)}{\varepsilon}\right)^{\frac{1}{\alpha}}\left(e^{-\lambda t}+e^{-\lambda(m_3-t)}\right)\right)^\alpha dt\\
		\le&\frac{4C^{3\alpha}}{\lambda\alpha}\cdot \frac{|a_{\mathcal{O}}|\|\psi\|_0+\|h\|_0}{\varepsilon}.
		\end{split}
		\end{align*}
		Therefore,
		\begin{align*}
		\begin{split}&\int_{0}^{ m_3}G(\phi_{t}(z))-G(\phi_{t}\phi_v(y_0))dt\\
		=&\int_{0}^{ m_3}\bar{u}(\phi_{t}(z))+\varepsilon d^\alpha(\phi_{t}(z),\mathcal{O})+h(\phi_{t}(z))-\bar{u}(\phi_{t+v}(y_0))-h(\phi_{t+v}(y_0))\\
		&-a_{\mathcal{O}}\big(\psi_\gamma(\phi_{t}(z))
		-\psi_\gamma(\phi_{t+v}(y_0))\big)dt\\
		\ge&\int_{0}^{ m_3}\bar{u}(\phi_{t}(z))-\bar{u}(\phi_{t+v}(y_0))+h(\phi_{t}(z))-h(\phi_{t+v}(y_0))\\
		&-a_{\mathcal{O}}\big(\psi_\gamma(\phi_{t}(z))
		-\psi_\gamma(\phi_{t+v}(y_0))\big)dt\\
		\ge& -(\|\bar u\|_\alpha+\|h\|_\alpha+|a_\mathcal{O}|\|\psi_\gamma\|_\alpha)\int_{0}^{ m_3}d^\alpha(\phi_{t}(z),\phi_{t+v}(y_0))dt\\
		\ge& - (\|\bar u\|_\alpha+\|h\|_\alpha+|a_\mathcal{O}|\|\psi_\gamma\|_\alpha) \cdot \frac{4C^{3\alpha}}{\lambda\alpha}\cdot \frac{|a_{\mathcal{O}}|\|\psi\|_0+\|h\|_0}{\varepsilon},
		\end{split}
		\end{align*}
		where we write $v$ short for $v(y_0,z)$. Also note that
		\begin{align*}
		|a_\mathcal{O}|
	=\left|\frac{\left\langle \mathcal{O},\bar u+h\right\rangle }{\left\langle \mathcal{O},\psi_\gamma\right\rangle}\right|\le \frac{ \|\bar u\|_0 + \|h\|_0 }{\psi_{min} }\le \frac{ \|\bar u\|_0 + 1}{\psi_{min} }.	
		\end{align*}
		Thus, one has that
			\begin{align}\label{X-6}
	\begin{split}&\int_{0}^{ m_3}G(\phi_{t}(z))-G(\phi_{t}\phi_v(y_0))dt\\
	\ge& - (\|\bar u\|_\alpha+\|h\|_\alpha+\frac{ \|\bar u\|_0 + 1 }{\psi_{min} }\|\psi_\gamma\|_\alpha) \cdot \frac{4C^{3\alpha}}{\lambda\alpha}\cdot \frac{|a_{\mathcal{O}}|\|\psi\|_0+\|h\|_0}{\varepsilon},
	\end{split}
	\end{align}
		Rewrite $m_3=p|\mathcal{O}|+r$ for some nonnegative integer $p$ and real number $0\le r\le |\mathcal{O}|$. By applying \eqref{X-2} and $\int Gd\mu_\mathcal{O}=0$, one has that
		\begin{align}\label{X-7}
		\begin{split}
		\int_{0}^{ m_3}G(\phi_t\phi_v(y_0))dt&=\int_{m_3-r}^{m_3}G(\phi_t\phi_v(y_0))dt\ge -|\mathcal{O}|\cdot (|a_{\mathcal{O}}|\|\psi\|_0+\|h\|_0).\\
		\end{split}
		\end{align}
		Combining \eqref{X-1}, \eqref{X-5}, \eqref{X-12}, \eqref{X-6} and \eqref{X-7}, we have
		\begin{align*}
		&\int_{0}^{ m_2}G(\phi_t(z))dt\ge \int_{0}^{ m_3}G(\phi_t(z))dt+\int_{m_2-1}^{ m_2}G(\phi_t(z))dt &\text{by }(\ref{X-12})\\
		=& \int_{0}^{ m_3}\left(G(\phi_t(z))-G(\phi_{t+v}(y_0))\right)dt\\
		&+\int_{0}^{ m_3}G(\phi_{t+v}(y_0))dt+\int_{m_2-1}^{ m_2}G(\phi_t(z))dt\\
		\ge&-(\|\bar u\|_\alpha+\|h\|_\alpha+\frac{ \|\bar u\|_0 + 1}{\psi_{min} }\|\psi_\gamma\|_\alpha) \cdot \frac{4C^{3\alpha}}{\lambda\alpha}\cdot \frac{|a_{\mathcal{O}}|\|\psi\|_0+\|h\|_0}{\varepsilon}&\text{by }(\ref{X-6})\\
		&-|\mathcal{O}|\cdot (|a_{\mathcal{O}}|\|\psi\|_0+\|h\|_0)&\text{by }(\ref{X-7})\\
		&+\varepsilon\left(\frac{D(\mathcal{O})}{4C^3e^{2\beta}}\right)^\alpha-|a_{\mathcal{O}}|\|\psi\|_0-\|h\|_0&\text{by }(\ref{X-5})\\
		\ge&\varepsilon\left(\frac{D(\mathcal{O})}{4C^3e^{2\beta}}\right)^\alpha\\
		&-\left(\frac{4C^{3\alpha}(\|\bar u\|_\alpha+\|h\|_\alpha+\frac{ \|\bar u\|_0 + 1 }{\psi_{min} }\|\psi_\gamma\|_\alpha) }{\lambda\alpha\varepsilon|\mathcal{O}|}+1+\frac{1}{|\mathcal{O}|}\right) \frac{\|\bar u\|_\alpha\|\psi\|_0}{\psi_{min}} d_{\alpha,Z_{u,\psi}}(\mathcal{O})\\
		&-\left(\frac{4C^{3\alpha}(\|\bar u\|_\alpha+\|h\|_\alpha+\frac{ \|\bar u\|_0 + 1}{\psi_{min} }\|\psi_\gamma\|_\alpha) }{\lambda\alpha\varepsilon}+|\mathcal{O}|+1\right) \frac{\|\psi\|_0+\psi_{min}}{\psi_{min}} \|h\|_0&\text{by }(\ref{X-1})\\
		\ge&\varepsilon\left(\frac{D(\mathcal{O})}{4C^3e^{2\beta}}\right)^\alpha\\
		&-\left(\frac{4C^{3\alpha}(\|\bar u\|_\alpha+10\varepsilon+\frac{ \|\bar u\|_0 + 1 }{\psi_{min} }\|\psi_\gamma\|_\alpha) }{\lambda\alpha\varepsilon \tau_0}+1+\frac{1}{\tau_0}\right) \frac{\|\bar u\|_\alpha\|\psi\|_0}{\psi_{min}} d_{\alpha,Z_{u,\psi}}(\mathcal{O})\\
		&-\left(\frac{4C^{3\alpha}(\|\bar u\|_\alpha+10\varepsilon+\frac{ \|\bar u\|_0 + 1}{\psi_{min} }\|\psi_\gamma\|_\alpha) }{\lambda\alpha\varepsilon}+|\mathcal{O}|+1\right) \frac{\|\psi\|_0+\psi_{min}}{\psi_{min}} \|h\|_0\\
		>&0,
		\end{align*}
	where we used Remark \ref{rem-2-4} and condition \eqref{condition}. Therefore, $m=m_2$ is the time as required since $m_2\ge 1\ge \tau$. This ends the  proof of Claim 2.
	\end{proof}

\subsection{Proof of Part II) of Theorem \ref{T:MainResult}}\label{S:ProofPartIIMainThm}
Firstly, we state a technical result on function approximation, which plays a key role in proving Proposition \ref{prop-3}. Proposition \ref{prop-3} can be viewed as a $C^1$-version of Proposition \ref{prop-2} which implies the part II) of Theorem \ref{T:MainResult}.
\begin{thm}[\cite{GW79}]\label{thm-G} Let $M$ be a smooth compact manifold. Then $\mathcal{C}^\infty(M)\cap \mathcal{C}^{0,1}(M)$ is Lip-dense in $\mathcal{C}^{0,1}(M)$.
\end{thm}
\begin{rem}In this theorem, {\it $\mathcal{C}^\infty(M)\cap \mathcal{C}^{0,1}(M)$ is Lip-dense in $\mathcal{C}^{0,1}(M)$} means that for any $g_1\in \mathcal{C}^{0,1}(M)$ and $\varepsilon>0$ there is corresponding $g_2\in \mathcal{C}^\infty(M)$ such that $\|g_1-g_2\|_0<\varepsilon$ and $\|g_2\|_1<\varepsilon+\|g_1\|_1$. Especially, $\|D_Mg_2\|_0< \varepsilon+\|g_1\|_1$, where $D_Mg$ is the derivative of function $g$ with respect to space variables.
\end{rem}
\begin{prop}\label{prop-3}

	Given $0<\varepsilon\le 1$, a strictly positive  $\psi\in\mathcal{C}^{1,0}(M)$ and $u\in \mathcal{C}^{1,0}(M)$, if a periodic segment $\mathcal{O}$ of $\Phi|_\L$ satisfies the following comparison condition
	
	\begin{align*}
	D(\mathcal{O})>
	\Big(\frac{4C^{3}(\|\bar u\|_1+10\varepsilon+\frac{ \|\bar u\|_0 + 1 }{\psi_{min} }\|\psi_\gamma\|_1) }{\l\varepsilon \tau_0}+1+\frac{1}{\tau_0}\Big) \frac{\|\bar u\|_1\|\psi\|_0}{\psi_{min}}\cdot\frac{ 400C^3e^{2\beta}}{\varepsilon}\cdot d_{1,Z_{u,\psi}}(\mathcal{O}),
	\end{align*}
	where  $\bar u$ is defined in Remark \ref{revael} and $\tau_0$ is the constant in Lemma \ref{M-0}, then there is a function  $w\in C^\infty(M)$ with $\|w\|_0<2\varepsilon\cdot diam(M)$ and $\|D_Mw\|_0<2\varepsilon$  such that the probability measure  $$\mu_{\mathcal{O}}\in\mathcal{M}_{min}(u+w+h;\psi,\L,\Phi),$$ where $h$ is any $C^1$ function with $\|D_Mh\|_0<5\varepsilon$ and
	$$\|h\|_0<\frac{1}{2}\cdot\min\left\{\frac{\frac{\varepsilon}{2}\left(\frac{D(\mathcal{O})}{4C^3e^{2\beta}}\right)}{\left(\frac{4C^{3}(\|\bar u\|_1+10\varepsilon+\frac{ \|\bar u\|_0 + 1 }{\psi_{min} }\|\psi_\gamma\|_1) }{\lambda\varepsilon}+|\mathcal{O}|+1\right) \frac{\|\psi\|_0+\psi_{min}}{\psi_{min}}},1\right\}.$$
\end{prop}
\begin{proof}By Theorem \ref{thm-G}, there exists a function $ w\in C^\infty$ such that
	$$\|D_Mw\|_0< \|\varepsilon d(\cdot,\mathcal{O})\|_1+\varepsilon\le 2\varepsilon$$
	and
	$$\| w-\varepsilon d(\cdot,\mathcal{O})\|_0<\min\left(\frac{H}{2},\varepsilon\cdot diam(M)\right),$$
	where $$H=\min\left\{\frac{\frac{\varepsilon}{2}\left(\frac{D(\mathcal{O})}{4C^3e^{2\beta}}\right)}{\left(\frac{4C^{3}(\|\bar u\|_1+10\varepsilon+\frac{ \|\bar u\|_0 + 1 }{\psi_{min} }\|\psi_\gamma\|_1) }{\lambda\varepsilon}+|\mathcal{O}|+1\right) \frac{\|\psi\|_0+\psi_{min}}{\psi_{min}}},1\right\}.$$
Next we show that $w$ is the function as required. Note that
	$$u+w+h=u+\varepsilon d(\cdot,\mathcal{O})+(w-\varepsilon d(\cdot,\mathcal{O})+h).$$
	Notice that, $$\|w-\varepsilon d(\cdot,\mathcal{O})+h\|_1\le\|D_Mw\|_0+\|\varepsilon d(\cdot,\mathcal{O})\|_1+\|h\|_1\le 2\varepsilon+\varepsilon+5\varepsilon<10\varepsilon,$$
	and
	$$\|w-\varepsilon d(\cdot,\mathcal{O})+h\|_0\le\|w-\varepsilon d(\cdot,\mathcal{O})\|_0+\|h\|_0< \frac{H}{2}+\frac{H}{2}=H.$$
	Then by Proposition \ref{prop-2}, we have that $\mu_{\mathcal{O}}\in\mathcal{M}_{min}(u+w+h;\psi,\L,\Phi).$ Additionally,
	$$\|w\|_0<\|\varepsilon d(\cdot,\mathcal{O})\|_0+\varepsilon\cdot diam(M)\le 2\varepsilon\cdot diam(M).$$
	This ends the proof.
\end{proof}
\begin{rem}\label{remark-2} Let $\widetilde w\in\mathcal{C}^{1,0}(M)$ be such  that $\|\widetilde{w}\|_{1,0}<\varepsilon$, $\widetilde{w}|_\mathcal{O}=0$ and $\widetilde w|_{M\setminus\mathcal{O}}>0$.  Then $\mu_\mathcal{O}$ is the unique measure in $\mathcal{M}_{min}(u+\widetilde w+w+h;\psi,\L,\Phi)$  whenever $\|h\|_1<5\varepsilon$ and
	$\|h\|_0$ is sufficiently small. The Proposition shows that there is an open set of $\mathcal{C}^{1,0}({M})$ near $u$ such that functions in the open set have the same unique minimizing measure with respect to $\psi$ and the measure supports on a periodic orbit.
\end{rem}

\section{Proofs of Technical Lemmas}\label{S:ProofsTechLem}
Note that throughout this section,  $\d,\e,\l,\b,C,\e',\d'$ are same as the ones in Remark \ref{remark-1-1}.
\subsection{Proof of Lemma \ref{An-1}}\label{S:LemAn-1}
	\begin{proof}We put a small positive constant $\tau$ with $\tau\ll1$ such that $|s(t_1)-s(t_2)|\le \eta$ for all $|t_1-t_2|\le\tau$ and $t_1,t_2\in[0,T].$ Since $\eta\le \frac{C^{10}e^{10\beta+10\lambda}}{e^{\lambda}-1}\delta'$, for all $0\le t\le T,$ there exists $r(t)$ with $|r(t)|<C\eta$ such that
	\begin{align}\label{l-1}
	w(\phi_{t+s(t)+r(t)}(y),\phi_{t}(x))=W^s_{\epsilon'}(\phi_{t+s(t)+r(t)}(y))\cap W^u_{\epsilon'}(\phi_{t}(x)).
	\end{align}
	Then for $t'\in[-\tau,\tau]$ and $t\in[\tau,T-\tau]$,
	$$\phi_{t'}(w( \phi_{t+s(t)+r(t)}(y),\phi_{t}(x)))\in W^s_{\epsilon'}(\phi_{t+s(t)+r(t)+t'}(y))\cap W^u_{\epsilon'}(\phi_{t+t'}(x)).$$
	On the other hand, one has that
	$$w(\phi_{t+t'+s(t+t')+r(t+t')}(y),\phi_{t+t'}(x))= W^s_{\epsilon'}(\phi_{t+t'+s(t+t')+r(t+t')}(y))\cap W^u_{\epsilon'}(\phi_{t+t'}(x)).$$
	Since $$|(t+t'+s(t+t')+r(t+t'))-(t+s(t)+r(t)+t')|\le (2C+1)\eta\ll\d,$$
	by the uniqueness of $v(\phi_{t+s(t)+r(t)+t'}y,\phi_{t+t'}x)$  given by (1) of the {\bf basic canonical setting}, one has that
	$$t+t'+s(t+t')+r(t+t')=t+s(t)+r(t)+t',$$
	and
	$$w(\phi_{t+t'+s(t+t')+r(t+t')}(y),\phi_{t+t'}(x))=\phi_{t'}(w( \phi_{t+s(t)+r(t)}(y),\phi_{t}(x))),$$
	for all $t'\in[-\tau,\tau]$ and $t\in[\tau,T-\tau]$.
	Since $\tau$ can be taken arbitrarily small, one has the following by induction
	\begin{align}\label{l-2}
	s(t)+r(t)=s(\tau)+r(\tau)=s(0)+r(0)=r(0)=v(y,x),\ \forall t\in[0,T].
	\end{align}
	Thus $$|s(t)|\le |r(t)|+|r(0)|\le 2C\eta,$$	
	and	 for all $t\in[\tau,T-\tau]$ and $t'\in[-\tau,\tau]$
	$$w(\phi_{t+t'+v(y,x)}(y),\phi_{t+t'}(x))=\phi_{t'}(w(\phi_{t+v(y,x)}(y),\phi_{t}(x))),$$
	which implies that
	\begin{align}\label{l-3}
w(\phi_{t+v(y,x)}(y),\phi_{t}(x))=\phi_{t}(w(\phi_{v(y,x)}(y),x))\ \forall t\in[0,T].
	\end{align}
	
	Now, we prove Ash 2). Note $w=w(\phi_{v(y,x)}(y),x)=W_{\epsilon'}^s(\phi_{v(y,x)}(y))\cap W_{\epsilon'}^u(x)$. Then by \eqref{l-1}, \eqref{l-2} and \eqref{l-3}, one has
	$$\phi_t(w)=W^s_{\epsilon'}(\phi_{t}\phi_{v(y,x)}(y))\cap W^u_{\epsilon'}(\phi_{t}(x))\text{ for all }t\in[0,T].$$
	Thus, for all $t\in[0,T]$, by (b) of the {\bf basic canonical setting}, 
	$$d(\phi_t(w),\phi_{t}\phi_{v(y,x)}(y))<Cd(\phi_{t}(x),\phi_{t}(y))\text{ and }d(\phi_t(w),\phi_{t}(x))<Cd(\phi_{t}(x),\phi_{t}(y)).$$
	Therefore
	$$d(\phi_t(w),\phi_t(x))\le Ce^{-\lambda(T-t)}d(\phi_T(w),\phi_T(x))\le C^2e^{-\lambda(T-t)}d(\phi_T(x),\phi_{T}(y)),$$
	where we used $w\in W_{\epsilon'}^u(x)$ and
	$$d(\phi_t(w),\phi_t\phi_{v(y,x)}(y))\le Ce^{-\lambda(t)}d(w,\phi_{v(y,x)}(y))\le C^2e^{-\lambda(t)}d(x,y).$$
	where we used $w\in W_{\epsilon'}^s(\phi_{(y,x)}v(y))$. By summing up, we have
	$$d(\phi_t\phi_{v(y,x)}(y),\phi_t(x))\le C^2e^{-\lambda\min(t,T-t)}(d(x,y)+d(\phi_T(x),\phi_T(y)))\text{ for } 0\le t\le T.$$
	
	Now we assume that $d(\phi_{t+s(t)}(y),\phi_t(x))\le \eta$ for all $t\ge 0,$ then by the arguments above.
	We have  for all $t\ge 0$,
	\begin{align*}d(\phi_t\phi_{v(y,x)}(y),\phi_t(x))&\le C^2e^{-\lambda\min(t,2t-t)}(d(x,y)+d(\phi_{2t}(x),\phi_{2t}(y)))\\
	&\le 2C^2\eta e^{-\lambda\min(t,2t-t)}\to 0\text{ as }t\to+\infty.
	\end{align*}
	This ends the proof.
\end{proof}
\subsection{Proof of Lemma \ref{lemma-a}}\label{S:Lemma-a}

\begin{proof}We partially follow Bowen's arguments in \cite{Bowen}. Firstly we fix a constant $K\gg C$ with $2C^2e^{-\lambda K}\ll1$ and a segment $\mathcal{S}$ as in Lemma \ref{lemma-a}. We let $\tau=|\mathcal{S}|$ and $\eta=d(\mathcal{S}^L,\mathcal{S}^R)$. Then $\eta<\delta'$ and $2Ce^{-\lambda\tau}\ll1$. Therefore, we have the following claim.
	
		\noindent{\it {\bf Claim A. } For the segment $\mc S$ in Lemma \ref{lemma-a}, there is a $y\in \L$ and a continuous  function $\hat{s}:\mathbb{R}\to \mathbb{R}$ with $\hat{s}(0)=0$ and $Lip(\hat{s})\le\frac{2C\eta}{\tau}$ such that $d(\phi_{i\tau+t_1+s(i\tau+t_1)}(y), \phi_{t_1}(\mathcal{S}^L))\le L_1\eta$ for all $t_1\in[0,\tau]$ and $i\in\mathbb{Z}$ where $L_1=2C^2(\frac{2}{e^\lambda-1}+e^{\beta}+2)$.}
		
	Since the proof of Claim A is long, we postpone the proof of Claim A to the next subsection. Let $y\in M$ and $\hat{s}:\mathbb{R}\to \mathbb{R}$ be as in Claim A.  We divide the following proof into two steps.
	
		\noindent {\it {\bf Step 1. } At first, we show that $y$ is a periodic point.}
		
		By Claim A,
	\begin{align*}
	d(\phi_{t+\hat{s}(t)}(y),\phi_{t+\tau+\hat{s}(t+\tau)}(y))\le 2L_1\eta\text{ for } t\in\mathbb{R}.	
	\end{align*}
	Since $Lip(\hat{s})\ll 1$ and $\hat s(0)=0$, $g(t)=t+\hat{s}(t)$ is a homomorphism of $\mathbb{R}$ onto itself, the above inequality can be rewritten as the following
	\begin{align*}
	d(\phi_{t}(y),\phi_{g^{-1}(t)+\tau+\hat{s}(g^{-1}(t)+\tau)}(y))\le 2L_1\eta\text{ for } t\in\mathbb{R}.		
	\end{align*}
	We note $$y'=\phi_{g^{-1}(0)+\tau+\hat{s}(g^{-1}(0)+\tau)}(y)$$ and $$s(t)=g^{-1}(t)+\hat{s}(g^{-1}(t)+\tau)-g^{-1}(0)-\hat{s}(g^{-1}(0)+\tau)-t.$$ Then
	$$d(\phi_{t}(y),\phi_{t+s(t)}(y'))\le 2L_1\eta\text{ for } t\in\mathbb{R} \text{ and } s(0)=0.$$
	Therefore, by Lemma \ref{An-1}, one has
	$$\phi_{v_2}(y')=y \text{ and }|v_2|\le 2CL_1\eta,$$
	where $v_2=v(y',y)$. Thus
	$$\phi_{g^{-1}(0)+\tau+\hat{s}(g^{-1}(0)+\tau)+v_2}(y)=y.$$
	Notice that $g^{-1}(0)=0$ since $g(0)=0$. Thus,
	$$|g^{-1}(0)+\hat{s}(g^{-1}(0)+\tau)+v_2|\le |\hat{s}(\tau)|+|v_2|\le (2C+2CL_1)\eta\ll \tau.$$
Therefore, $y$ is a periodic point.
	
\noindent{{\bf Step 2.} There is a periodic segment $\mathcal{O}$ such that $$ ||\mathcal{S}|-|\mathcal{O}||\le Ld(\mathcal{S}^L, \mathcal{S}^R)$$ and
$$ d(\phi_t(\mathcal{O}^L),\phi_t(\mathcal{S}^L))\le Ld(\mathcal{S}^L, \mathcal{S}^R)\text{ for all } 0\le t\le \max(|\mathcal{S}|,|\mathcal{O}|),$$
	where $L=2CL_2+L_2$ and $L_2=2C^3L_1+C^4L_1.$ }
	
	By Claim A,
	\begin{align*}
	d(\phi_{t+\hat{s}(t)}(y),\phi_{t}(\mathcal{S}^L))\le L_1\eta\text{ for } t\in[0,\tau].	
	\end{align*}
	By Lemma \ref{An-1},
	for $t\in [0,\tau],$ $|\hat{s}(t)|\le 2CL_1\eta$ and
	\begin{align}\label{eq-ano-1}\begin{split}
	d(\phi_t\phi_{v_1}(y),\phi_t(\mathcal{S}^L))&\le  C^2e^{-\lambda\min(t,\tau-t)}(d(y,\mathcal{S}^L)+d(\phi_\tau(y),\phi_\tau(\mathcal{S}^L)))\\
	&\le C^2(d(y,\mathcal{S}^L)+d(\phi_{\tau+\hat{s}(\tau)}(y),\phi_\tau(\mathcal{S}^L))+d(\phi_{\tau+\hat{s}(\tau)}(y),\phi_{\tau}(y)))\\
	&\le L_2\eta,
	\end{split}\end{align}
	where $v_1=v(y,\mc S^L)$. Now we put $y^*=\phi_{v_1}y$ and we have a periodic segment,
	$$\mathcal{O}:[0,\tau+g^{-1}(0)+\hat{s}(g^{-1}(0)+\tau)+v_2]\to M: t\to \phi_t(y^*),$$
	where $v_2$ is as in  Step 1.
	
	It is clear that
	$$||\mathcal{S}|-|\mathcal{O}||\le |g^{-1}(0)+\hat{s}(g^{-1}(0)+\tau)+v_2|\le L_2\eta.$$
	If $|\mathcal{O}|\le|\mathcal{S}|$, by \eqref{eq-ano-1},
	$$d(\phi_t(y^*),\phi_t(\mathcal{S}^L))\le  L_2\eta,\text{ for }t\in[0,\max(|\mathcal{S}|,|\mathcal{O}|)],$$
	where we used $\tau=\max(|\mathcal{S}|,|\mathcal{O}|).$\\
	If  $|\mathcal{O}|>|\mathcal{S}|$, by \eqref{eq-ano-1},
	$$d(\phi_t(y^*),\phi_t(\mathcal{S}^L))\le  L_2\eta,\text{ for }t\in[0,|\mathcal{S}|],$$
	and for $t\in(|\mathcal{S}|,|\mathcal{O}|]$,
	\begin{align*}d(\phi_t(y^*),\phi_t(\mathcal{S}^L))&\le  d(\phi_t(y^*),\phi_\tau(y^*))+d(\phi_\tau(y^*),\phi_\tau(\mathcal{S}^L))+d(\phi_\tau(\mathcal{S}^L),\phi_t(\mathcal{S}^L))\\
	&\le L\eta,\end{align*}
	where $L=2CL_2+L_2$.
	This ends the proof since $L_2\le L$.
\end{proof}
\subsubsection{Proof of Claim A}
\begin{proof} Recall that $\mathcal{S}$ is a segment of $\Phi|_\L$ with $|\mathcal{S}|=\tau\ge K$ and $d(\mathcal{S}^L,\mathcal{S}^R)=\eta<\delta'$ where $K$ satisfies $2C^2e^{-\lambda K}\ll 1$. We define $x_{-k}, \zeta_{-k}$ recursively
	for $k\ge 0$  by $$x_0=\mathcal{S}^R, \zeta_{0}=0$$ and
	$$ \zeta_{-k-1}=v(\phi_{-\tau}(x_{-k}),\mathcal{S}^R), x_{-k-1}=W_{\epsilon'}^s(\phi_{-\tau+\zeta_{-k-1}}(x_{-k}))\cap W_{\epsilon'}^u(\mathcal{S}^R) \text{ for }k=1,2,\cdots.$$
	We have the following two assertions.
	
	\noindent {\it {\bf Assertion 1.}  $x_{-k}$ and $\zeta_{-k}$ are well defined and $d(x_{-k},\mathcal{S}^R)\le 2C\eta$ for $k\ge 0.$}
	\begin{proof}	
		In the case $k=0$, it is obviously true. Now assume that we have $\zeta_{-k},x_{-k}$ and $d(x_{-k},\mathcal{S}^R)\le 2C\eta$.
		Then
		\begin{align}\begin{split}\label{E:XkSR}
		d(\phi_{-\tau}(x_{-k}),\mathcal{S}^R)&\le d(\phi_{-\tau}(x_{-k}),\phi_{-\tau}(\mathcal{S}^R))+d(\mathcal{S}^R,\phi_{-\tau}(\mathcal{S}^R))\\
		&\le Ce^{-\lambda \tau}\cdot 2C\eta+\eta\\
		&\le 2\eta,
		\end{split}
		\end{align}
		where we used $x_{-k}\in W_{\epsilon'}^u(\mathcal{S}^R)$.
		Since $2\eta\le 2\delta'$, $x_{-k-1}$ is well defined as well as $\zeta_{-k-1}$, and moreover one has that
		\begin{align*}
		d(x_{-k-1},\mathcal{S}^R)\le 2C\eta.
		\end{align*}
		This ends the proof.
	\end{proof}
	By (\ref{E:XkSR}), we have that
	\begin{align}\label{eq-an-1}
	|\zeta_k|\le 2C\eta\ll 1.
	\end{align}
	Next we denote $x^{(-k)}=\phi_{k\tau-\sum_{i=0}^{k}\zeta_{-i}}(x_{-k})$ for $k\ge 0$. For $k\in\mathbb{N}$, we define $s^*_{-k}:\mathbb{R}\to\mathbb{R}$ by
	\[s^*_{-k}(t)=\left\{\begin{array}{ll}
	\zeta_0, & \text{ if } t> 0,\\
	\sum_{i=0}^{l-1}\zeta_{-i}, & \text{ if } -l\tau< t\le-(l-1)\tau, l\in\{1,2,\cdots,k\},\\
	\sum_{i=0}^{k}\zeta_{-i}, & \text{ if } t\le-k\tau.
	\end{array}
	\right.\]

	\noindent {\bf Assertion 2.} {\it There exists a constant $L_0$ such that for $t=-j\tau-t_0$ satisfying  $ t_0\in[0,\tau)$ and $j\in\{0,1,\cdots,k-1\}$, the following holds}
$$d\left(\phi_{t+s^*_{-k}(t)}\left(x^{(-k)}\right),\phi_{-t_0}\left(\mathcal{S}^R\right)\right)\le L_0\eta.$$	
	\begin{proof}We fix $t=-j\tau-t_0$ for some $j\in\{0,1,\cdots,k-1\}$ and $ t_0\in[0,\tau)$. Since $x_{-j}\in W^{u}_{\epsilon'}(\mathcal{S}^{R}) $, we have
		\begin{align}\label{eq-an-2}
		d\left(\phi_{-t_0}\left(x_{-j}\right),\phi_{-t_0}\left(\mathcal{S}^R\right)\right)\le Ce^{-\lambda t_0}d\left(x_{-j},\mathcal{S}^R\right)\le 2C^2\eta.
		\end{align}
				Note that $\tau-\zeta_{-j-1}-t_0\ge -1$ and $x_{-j-1}\in W^{s}_{\epsilon'}(\phi_{-\tau +\zeta_{-j-1}}(x_{-j}))\cap W^{u}_{\epsilon'}(\mathcal{S}^{R})$, we have
		\begin{align}\label{eq-an-3}
		\begin{split}
		&d\left(\phi_{\tau-\zeta_{-j-1}-t_0}\left(x_{-j-1}\right),\phi_{-t_0}\left(x_{-j}\right)\right)\\
		=&d\left(\phi_{\tau-\zeta_{-j-1}-t_0}\left(x_{-j-1}\right),\phi_{\tau-\zeta_{-j-1}-t_0}\phi_{-\tau+\zeta_{-j-1}}\left(x_{-j}\right)\right)\\
		\le& e^{\beta }d\left(x_{-j-1},\phi_{-\tau+\zeta_{-j-1}}\left(x_{-j}\right)\right)\\
				\le&2C^2e^{\beta}\eta,
		\end{split}\end{align}
		where we used  (2) and (3) of the {\bf basic canonical setting} for the case $\tau-\zeta_{-j-1}-t_0> 0$ and $\tau-\zeta_{-j-1}-t_0\le 0$, respectively.
		
		Note that $|\zeta_{-l}|\ll 1, \tau\gg 1$ and $t_0\in[0, \tau)$, i.e., $\tau-\zeta_{-j-1}-t_0\geq -1$ and $\tau -\zeta_{-i} >\tau-1>1$. Since $x_{-(k-l)}\in W^{s}_{\epsilon'}(\phi_{-\tau+\zeta_{-(k-l)}}(x_{-(k-j-1)}))\cap W^{u}_{\epsilon'}(\mathcal{S}^{R})$, we have
		\begin{align}\begin{split}\label{eq-an-4}
		&\sum_{l=0}^{k-j-2}d\left(\phi_{(k-j-l-1)\tau-\sum_{i=j+1}^{k-l}\zeta_{-i}-t_0}\left(x_{-(k-l)}\right),\phi_{(k-j-l-2)\tau-\sum_{i=j+1}^{k-l-1}\zeta_{-i}-t_0}\left(x_{-(k-l-1)}\right)\right)\\
		&\le \sum_{l=0}^{k-j-2} Ce^{-\lambda\left(\left(k-j-l\right)\tau-\sum_{i=j+1}^{k-l}\zeta_{-i}-t_0\right)}d\left(x_{-(k-l)},\phi_{-\tau+\zeta_{-(k-l)}}\left(x_{-(k-j-1)}\right)\right)\\
		&\le  \sum_{l=0}^{k-j-2} Ce^{-\lambda\left(\left(k-j-l\right)\tau-\sum_{i=j+1}^{k-l}\zeta_{-i}-t_0\right)} 4C\eta\\
		&\le 4C^2\eta e^{-\lambda (2\tau - \zeta_{-j-1}-\zeta_{-j-2}-t_0)}\sum_{l=0}^{k-l-1}e^{-\lambda\cdot (\tau -1 )\cdot l}\\
		&\le 4C^2\eta e^{-\lambda(\tau-2)}\frac{1}{1-e^{-\lambda}}\\
		&=4C^2\eta \frac{1}{e^{\lambda}-1},
		\end{split}	\end{align}
where we used 1b) of the {\bf basic canonical setting}.
Combining \eqref{eq-an-2}, \eqref{eq-an-3} 
and \eqref{eq-an-4}, we have that for $t=-j\tau-t_0$
		\begin{align*}
		&d\left(\phi_{t+s^*_{-k}(t)}\left(x^{(-k)}\right),\phi_{-t_0}\left(\mathcal{S}^R\right)\right)\\
		=&d\left(\phi_{-j\tau+\sum_{i=0}^{j}\zeta_{-i}-t_0}\left(x^{(-k)}\right),\phi_{-t_0}\left(\mathcal{S}^R\right)\right)\\
		=&d\left(\phi_{(k-j)\tau-\sum_{i=j+1}^{k}\zeta_{-i}-t_0}\left(x_{-k}\right),\phi_{-t_0}\left(\mathcal{S}^R\right)\right)\\
		\le&\sum_{l=0}^{k-j-2}d\left(\phi_{(k-j-l-1)\tau-\sum_{i=j+1}^{k-l}\zeta_{-i}-t_0}\left(x_{-(k-l)}\right),\phi_{(k-j-l-2)\tau-\sum_{i=j+1}^{k-l-1}\zeta_{-i}-t_0}\left(x_{-(k-l-1)}\right)\right)\\
		&+d(\phi_{\tau-\zeta_{-j-1}-t_0}(x_{-j-1}), \phi_{-t_0}(x_{-j}))+d\left(\phi_{-t_0}(x_{-j}),\phi_{-t_0}(\mathcal{S}^R)\right)\\
		\le &L_0\eta,
		\end{align*}
		where $L_0=2C^2(\frac{2}{e^\lambda-1}+e^{\beta}+1)$. This ends the proof of Assertion 2.
	\end{proof}

	Now for $k\in\mathbb{N}$, we define $\bar{s}_{-k}:\mathbb{R}\to\mathbb{R}$ by
	\[\bar{s}_{-k}(t)=\left\{\begin{array}{ll}
	\zeta_0, & \text{ if } t> 0,\\
	\sum_{i=0}^{l-1}\zeta_{-i}-\frac{t+(l-1)\tau}{\tau}\zeta_{-l}, & \text{ if } -l\tau<  t\le -(l-1)\tau, l\in\{1,2,\cdots,k\},\\
	\sum_{i=0}^{k}\zeta_{-i}, & \text{ if } t\le -k\tau.
	\end{array}
	\right.\]
	It is clear that $\bar{s}_{-k}$ is Lipschitz continuous with
	$$Lip(\bar{s}_{-k})\le \frac{\max_{i\in\{0,1,2,\cdots,k\}}|\zeta_i|}{\tau}\le \frac{2C\eta}{\tau},$$
	and
	$$|\bar{s}_{-k}(t)-s^*_{-k}(t)|\le\max_{i\in\{0,1,2,\cdots,k\}}|\zeta_i|\le 2C\eta.$$
	Therefore, by {Assertion 2},  when $t=-j\tau-t_0$ for some $j\in\{0,1,2,\cdots,k-1\}$ and $ t_0\in[0,\tau)$, one has that
	\begin{align}\begin{split}\label{E:BarSEst}
	&d\left(\phi_{t+\bar{s}_{-k}(t)}\left(x^{(-k)}\right),\phi_{-t_0}\left(\mathcal{S}^R\right)\right)\\
	\le& d\left(\phi_{t+\bar{s}_{-k}(t)}\left(x^{(-k)}\right),\phi_{t+s^*_{-k}(t)}\left(x^{(-k)}\right)\right)+d\left(\phi_{t+s^*_{-k}(t)}\left(x^{(-k)}\right),\phi_{-t_0}\left(\mathcal{S}^R\right)\right)\\
	\le& L_1\eta,	
	\end{split}\end{align}
	where $L_1=L_0+2C^2.$
	Now for $k\in\mathbb{N}$, we define ${s}_{-k}:\mathbb{R}\to\mathbb{R}$ by
	$${s}_{-k}(t)=\bar{s}_{-2k}(t-k\tau)-\sum_{i=0}^{k}\zeta_{-i}.$$
	It is clear that ${s}_{-k}(0)=0$.
	On the other hand, we note $y_k=\phi_{-\tau k+\sum_{i=0}^{k}\zeta_{-i}}(x^{(-2k)})$.
	Thus,  when $t=-j\tau-t_0$ for some $j\in\{-k,-k+1,\cdots,k-1\}$ and $ t_0\in[0,\tau)$, (\ref{E:BarSEst}) implies that
	\begin{align*}
	d\left(\phi_{t+{s}_{-k}(t)}(y_k),\phi_{-t_0}\left(\mathcal{S}^R\right)\right)=d\left(\phi_{t-\tau k+\bar{s}_{-2k}(t-\tau k)}\left(x^{(-2k)}\right),\phi_{-t_0}\left(\mathcal{S}^R\right)\right)\le L_1\eta.	
	\end{align*}
	Notice that $s_{-k}$ are Lipschitz with $Lip(s_{-k})\le \frac{2C\eta}{\tau}\ll \eta$ for all $k\in\mathbb{N}$. Applying the Ascoli-Azel\'a theorem,  there exists a subsequence $(s_{-k_i})_{i=1}^{+\infty}$ that converges to a Lipschitz continuous function $\hat{s}:\mathbb{R}\to\mathbb{R}$ with $Lip(\hat{s})\le \frac{2C\eta}{\tau}\ll \eta$ and $\hat {s}(0)=0$. Without losing any generality, we assume that $y_{k_i}\to y$ as $i\to+\infty$. By the continuity,  if $t=-j\tau-t_0$ for some $j\in\mathbb{Z}$ and $ t_0\in[0,\tau)$, then
	\begin{align*}
	d(\phi_{t+\hat{s}(t)}(y),\phi_{-t_0}(\mathcal{S}^R))\le L_1\eta.	
	\end{align*}
	 That is to say, if $t=-(j+1)\tau+(\tau-t_0)$ for some $j\in\mathbb{Z}$ and $ t_0\in[0,\tau)$, then
	\begin{align*}
	d\left(\phi_{t+\hat{s}(t)}(y),\phi_{\tau -t_0}(\mathcal{S}^L)\right)\le L_1\eta.	
	\end{align*}
	Note that $y_{k_i}\in \L$ for each $i\in\mb N$, thus $y\in\L$.
	Let $t_1=\tau -t_0$, then the proof of Claim A is completed.
\end{proof}

\subsection{Proof of Lemma \ref{MCGB}}\label{S:ProofMCGB}
In this section, we mainly prove a version of the so called  Ma\~n\`e-Conze-Guivarc'h-Bousch's Lemma. The proof partially follows Bousch's arguments in \cite{Bousch_Mane}.

\subsubsection{Integration along segment}
Recall that, for a continuous function $u$ and a segment $\mathcal{S}$ of $\Phi$, the integration of $u$ along $\mathcal{S}$ is defined by
$$\left\langle\mathcal{S},u\right\rangle:=\int_{a}^bu(\phi_t(x))dt.$$
\begin{lem}\label{Bou-2}
	Let $u:M\to\mathbb{R}$ be an $\alpha$-H\"older function with $\beta(u;1,\L,\Phi)\ge 0$. Then for a segment $\mathcal{S}$ of $\Phi|_\L$ satisfying $|\mathcal{S}|\ge K$ and $d(\mathcal{S}^L,\mathcal{S}^R)\le\delta'$, the following holds $$\left\langle\mathcal{S},u\right\rangle\ge-K_1d^\alpha(\mathcal{S}^L, \mathcal{S}^R),$$
	where $K_1=\frac{(CL)^\alpha}{\lambda\alpha}\|u\|_\alpha+L\|u\|_0$.
\end{lem}
\begin{proof}Since $d(\mathcal{S}^L,\mathcal{S}^R)\le\delta'$ and $|\mc S|\ge K$, by Anosov Closing Lemma, there exists a periodic segment $\mathcal{O}$ of $\Phi|_\L$ such that
 $$ ||\mathcal{S}|-|\mathcal{O}||\le Ld(\mathcal{S}^L, \mathcal{S}^R)$$ and
$$ d(\phi_t(\mathcal{O}^L),\phi_t(\mathcal{S}^L))\le Ld(\mathcal{S}^L, \mathcal{S}^R)\text{ for all } 0\le t\le \max(|\mathcal{S}|,|\mathcal{O}|).$$
Therefore, Letting $v=v(\mathcal{O}^L,\mathcal{S}^L)$ as in Lemma \ref{An-1}, one has that
\begin{align*}
\left\langle\mathcal{S},u\right\rangle-\left\langle\mathcal{O},u\right\rangle&=\int_0^{|\mathcal{O}|}u(\phi_t(\mathcal{S}^L))-u(\phi_t\phi_v(\mathcal{O}^L))dt+\int_{|\mathcal{O}|}^{|\mathcal{S}|}u(\phi_t(\mathcal{S}^L))dt\\
&\ge -\|u\|_\alpha\int_0^{|\mathcal{O}|}d^\alpha(u(\phi_t(\mathcal{S}^L)),u(\phi_t\phi_v(\mathcal{O}^L)))dt-\|u\|_0 ||\mathcal{S}|-|\mathcal{O}||\\
&\ge -\|u\|_\alpha\int_0^{|\mathcal{O}|}\big(Ce^{-\lambda\min(t,T-t)}Ld(\mathcal{S}^L, \mathcal{S}^R)\big)^\alpha dt-\|u\|_0Ld(\mathcal{S}^L, \mathcal{S}^R)\\
&\ge -\left(\frac{(CL)^\alpha}{\lambda\alpha}\|u\|_\alpha+L\|u\|_0\right)d^\alpha(\mathcal{S}^L, \mathcal{S}^R),
\end{align*}
where we used the assumption $0<\alpha\le 1$ and $0<d(\mathcal{S}^L, \mathcal{S}^R)<\delta'\ll1$. Then the Lemma is immediately from the fact $\left\langle\mathcal{O},u\right\rangle\ge 0$ since $\beta(u;1,\L,\Phi)\ge 0$.
This ends the proof.
\end{proof}

\begin{lem}\label{Bou-3}
	Let $\mathcal{P}$ be a finite partition of $M$ with diameter smaller than $\delta'$ and $u:M\to\mathbb{R}$ be an $\alpha$-H\"older function with $\beta(u;1,\L,\Phi)\ge 0$. Then for a given segment $\mathcal{S}$ of $\Phi|_\L$, the following holds
	\[\left\langle\mathcal{S},u\right\rangle\ge-K_2\delta'^\alpha,\]
	where $K_2=\sharp \mc P\cdot\left(\frac{K\|u\|_0}{\delta'^\alpha}+K_1\right)$ and $K$, $K_1$ are as in Lemma \ref{lemma-a} and \ref{Bou-2} respectively.
	\end{lem}
\begin{proof}For $x\in M$, denote $\mathcal{P}(x)$ the element in $\mathcal{P}$ which contains $x$. Assume $|\mathcal{S}|=(n-1)K+r$ for some $n\ge 1$ and $0\le r<K$. Let $t_i=iK$ for $0\le i\le n-1$ and $t_n=|\mathcal{S}|$.
	We define the function $w:\mathbb{N}\to[0,n]\cap\mathbb{N}$ inductively by letting
	\[w(0)=0\]
	\[w(k)=\min\{\eta(w(k-1))+1,n\}.\]
	where $\eta: [0, n-1]\cap\mathbb{N}\to[0, n-1]\cap\mathbb{N}$ is the function that maps each $i$ to the largest $j\in[0,n-1]\cap\mathbb{N}$ such that $\mathcal{P}(\phi_{t_i}(\mathcal{S}^L))=\mathcal{P}(\phi_{t_{j}}(\mathcal{S}^L))$. Let $s\ge 0$ be the smallest integer for which $\eta(w(s))=n-1$. Then $\mathcal{P}(\phi_{t_{w(i)}}(\mathcal{S}^L))\neq\mathcal{P}(\phi_{t_{w(j)}}(\mathcal{S}^L))$ for $0\le i<j\le s$ which implies $s\le\sharp\mathcal{P}$. For $0\le j\le s$, we have two cases:
	If $\eta(w(j))=w(j)$
	\begin{align}\label{B-1}
	\int_{t_{w(j)}}^{t_{\eta(w(j))}}u(\phi_t(x))dt=0\text{ and }\int_{t_{\eta(w(j))}}^{t_{\eta(w(j))+1}}u(\phi_t(x))dt\ge -K\|u\|_0.
	\end{align}
	If $\eta(w(j))>w(j)$, by using Lemma \ref{Bou-2},
	\begin{align}\label{B-2}
\int_{t_{w(j)}}^{t_{\eta(w(j))}}u(\phi_t(x))dt\ge-\left(\frac{(CL)^\alpha}{\lambda\alpha}\|u\|_\alpha+L\|u\|_0\right)\delta'^\alpha,
\end{align}
	where we use the fact $d(\phi_{t_i}(\mathcal{S}^L),\phi_{t_{j}}(\mathcal{S}^L))<\delta'$ since $\mathcal{P}(\phi_{t_i}(\mathcal{S}^L))=\mathcal{P}(\phi_{t_{j}}(\mathcal{S}^L))$.
	On the other hand, as in \eqref{B-1},
		\begin{align}\label{B-3}
	\int_{t_{\eta(w(j))}}^{t_{\eta(w(j))+1}}u(\phi_t(\mathcal{S}^L))dt\ge -K\|u\|_0.
	\end{align}
	Combining \eqref{B-1}, \eqref{B-2} and \eqref{B-3}, one has
	\begin{align*}
	&\left\langle\mathcal{S},u\right\rangle=\sum_{j=0}^{s-1}	\int_{t_{w(j)}}^{t_{\eta(w(j))}}+\int_{t_{\eta(w(j))}}^{t_{\eta(w(j))+1}}u(\phi_t(\mathcal{S}^L))dt\\
	\ge&-s\left(K\|u\|_0+\left(\frac{(CL)^\alpha}{\lambda\alpha}\|u\|_\alpha+L\|u\|_0\right)\delta'^\alpha\right)\\
	\ge&-\sharp \mc P\cdot\left(K\|u\|_0+\left(\frac{(CL)^\alpha}{\lambda\alpha}\|u\|_\alpha+L\|u\|_0\right)\delta'^\alpha\right),
	\end{align*}
which completes the proof.
\end{proof}

In the following, we deal with the so called shadowing property for  two finite time segments, which will allow one to use one segment to shadow two segments of which the ending point of one segment is close to the beginning point of the other. Let $\mc S_1$ and $\mc S_2$ be two segments of $\Phi|_\L$, suppose that
$$d(\mc S_1^R,\mc S_2^L)<\d'.$$
Then there exist $v(\mc S_2^L,\mc S_1^R)$ and $w(\mc S_2^L,\mc S_1^R)=W^s_{\e'}(\phi_{v(\mc S_2^L,\mc S_1^R)}(\mc S_2^L))\cap W^u_{\e'}(\mc S_1^R)$. Define a new segment $\mc S_1*\mc S_2:\left[-|\mc S_1|,|\mc S_2|-v(\mc S_2^L,\mc S_1^R)\right]$ by letting
\begin{equation}\label{E:S1*S2}
\mc S_1*\mc S_2(t)=\phi_t\left(w(\mc S_2^L,\mc S_1^R)\right)\ \forall t\in \left[-|\mc S_1|,|\mc S_2|-v(\mc S_2^L,\mc S_1^R)\right].
\end{equation}
We remark here that the definition of $\mc S_1*\mc S_2$ above is not the unique way for describing the shadowing property. Nevertheless, it is the most convenient way for the rest of the proof.
\begin{lem}\label{Bou-4}
Given $0<\alpha\le 1$ and a large constant $\gamma=\gamma(\alpha)\gg 1$  satisfying that $2C^{2\alpha}e^{-\frac{\gamma\alpha\lambda}{2}}\ll 1$, when two segments $\mathcal{S}_1$ and $\mathcal{S}_2$ of $\Phi|_\L$ satisfy  the following
$$d(\mathcal{S}_1^R,\mathcal{S}_2^L)\le \delta'\text{ and } \min\{|\mathcal{S}_1|,|\mathcal{S}_2|\}\ge \gamma,$$
then for all $u\in \mathcal{C}^{0,\alpha}(M)$,
	$$\frac{\left|\left\langle\mathcal{S}_1*\mathcal{S}_2,u\right\rangle-\left\langle\mathcal{S}_1,u\right\rangle-\left\langle\mathcal{S}_2,u\right\rangle\right|}{d^\a(\mathcal{S}_1^R,\mathcal{S}_2^L)-d^\a((\mathcal{S}_1*\mathcal{S}_2)^R,\mathcal{S}_2^R)-d^\a((\mathcal{S}_1*\mathcal{S}_2)^L,\mathcal{S}_1^L)}\le K_3,$$
	where $K_3=\frac{C\|u\|_0+\frac{2C^{2\alpha}\|u\|_\alpha}{\lambda\alpha}}{1-2C^{2\alpha}e^{-{(\gamma-1)\alpha\lambda}}}$ and the denominator of the left side of the above inequality is always positive by the choice of $\g$.
\end{lem}
\begin{proof}Fix $\alpha,\gamma,u,\mathcal{S}_1,\mathcal{S}_2$ as in this Lemma. Note $v=v(\mathcal{S}_2^L,\mathcal{S}_1^R)$, $w=w\left(\mathcal{S}_2^L,\mathcal{S}_1^R\right)$ and
	$$\widetilde{\mathcal{S}}_2:[0,|\mathcal{S}_2|-v]:t\to\phi_{t+v}(\mathcal{S}_2^L).$$
Thus, we have
	\begin{align*}
	&\left|\left\langle\mathcal{S}_1*\mathcal{S}_2,u\right\rangle-\left\langle\widetilde{\mathcal{S}}_2,u\right\rangle-\left\langle\mathcal{S}_1,u\right\rangle\right|\\
	=&\left|\int_0^{|\mathcal{S}_2|-v}u(\phi_{t}(w))-u(\phi_{t}\phi_v(\mathcal{S}_2^L))dt+\int_0^{|\mathcal{S}_1|}u(\phi_{-t}(w))-u(\phi_{-t}(\mathcal{S}_1^R))dt\right|\\
	\le&\int_0^{|\mathcal{S}_2|-v}\|u\|_\alpha d^\alpha(\phi_{t}(w),\phi_{t}\phi_{v}(\mathcal{S}_2^L))dt+\int_0^{|\mathcal{S}_1|}\|u\|_\alpha d^\alpha(\phi_{-t}(w),\phi_{-t}(\mathcal{S}_1^R))dt\\
	\le &\int_0^{|\mathcal{S}_2|-v}\|u\|_\alpha (Ce^{-\lambda t})^\alpha d^\a(w,\mc S_2^L)dt+\int_0^{|\mathcal{S}_1|}\|u\|_\alpha (Ce^{-\lambda t})^\alpha d^\a(w,\mc S_1^R)dt\\
	\le&2\|u\|_\alpha\frac{C^{2\alpha}}{\lambda\alpha}d^\a(\mc S_1^R,\mc S_2^L),
	\end{align*}
	and
	\begin{align*}
	\left|\left\langle\widetilde{\mathcal{S}}_2,u\right\rangle-\left\langle\mathcal{S}_2,u\right\rangle\right|\le \|u\|_0|v|\le \|u\|_0Cd(\mathcal{S}_1^R,\mathcal{S}_2^L)\le \|u\|_0Cd^\a(\mathcal{S}_1^R,\mathcal{S}_2^L).
	\end{align*}
	Therefore
	\begin{equation}\label{E:ShadowEst}
	\left|\left\langle\mathcal{S}_1*\mathcal{S}_2,u\right\rangle-\left\langle\mathcal{S}_1,u\right\rangle-\left\langle\mathcal{S}_2,u\right\rangle\right|\le \left(C\|u\|_0+\frac{2C^{2\alpha}\|u\|_\alpha}{\lambda\alpha}\right)d^\a(\mathcal{S}_1^R,\mathcal{S}_2^L)
	\end{equation}
	On the other hand, one has that
	$$d^\a((\mathcal{S}_1*\mathcal{S}_2)^L,\mathcal{S}_2^L)\le C^\a e^{-\a\lambda(\gamma-v)}d^\a(w,\mc S_2^L)\le C^{2\a}e^{-\a\lambda(\gamma-1)}d^\a(\mathcal{S}_1^R,\mathcal{S}_2^L),$$
	and $$d^\a((\mathcal{S}_1*\mathcal{S}_2)^R,\mathcal{S}_2^R)\le C^\a e^{-\a\lambda\gamma}d^\a(w,\mc S_1^R)\le C^{2\a}e^{-\a\lambda\gamma}d^\a(\mathcal{S}_1^R,\mathcal{S}_2^L),$$
	which combining with (\ref{E:ShadowEst}) and the choice of $\g$ implies what needed, thus accomplish the proof.
	\end{proof}
\subsubsection{Proof of Lemma \ref{MCGB}}\label{S:PfMCGB}
Before the main proof, we first state a technical Lemma which can be deduced from the Lemma 1.1 of \cite{Bousch_Mane}.
\begin{lem}\label{lemma-21}Given $0<\alpha\le1,A>0, \gamma\in\mathbb{R}$ and a continuous function $u:M\to\mathbb{R}$, the following are equivalent
	\begin{itemize}
		\item[(1).] For all $n\ge 1$ and $x_i\in M, {i\in\mathbb{Z}/n\mathbb{Z}}$,
		\begin{align}\label{2-12-1}\sum_{i\in\mathbb{Z}/n\mathbb{Z}}u(x_i)+A\sum_{i\in\mathbb{Z}/n\mathbb{Z}}d^\alpha(\phi_\gamma x_i,x_{i+1})\ge 0.\end{align}
		\item[(2).] There exists an $\alpha$-H\"older function $v:M\to\mathbb{R}$ with $\|v\|_\alpha\le A$ such that $u\ge v\circ \phi_\gamma-v$.
	\end{itemize}
\end{lem}
Now we prove  Lemma \ref{MCGB}.
\begin{proof}Let $K_1,K_2,K_3$ be the constants as in Lemmas \ref{Bou-2}, \ref{Bou-3} and \ref{Bou-4}. We fix a $\g>N_0$ satisfying the condition in Lemma \ref{Bou-4} and a  large number $Q$ such that
\[Q>\max\{K_1,K_2,K_3\}.\]

 For $n\ge 1$, we note $i^{(n)}=i+n\mathbb{Z}\in\mathbb{Z}/n\mathbb{Z}$ for $i\in[0,n-1]\cap\mathbb{Z}$. Now we fix an integer $n\ge 1$ and points $x_{i^{(n)}}\in \L,i^{(n)}\in\mathbb{Z}/n\mathbb{Z}$. Note  $$\mathcal{S}_{i^{(n)}}: [0,\g]\to \L:t\to\phi_t(x_{i^{(n)}})\text{ for }i^{(n)}\in\mathbb{Z}/n\mathbb{Z},$$ $$\mathcal{L}^{(n)}=\{\mathcal{S}_{i^{(n)}},i^{(n)}\in\mathbb{Z}/n\mathbb{Z}\}$$
and \begin{align*}
\Sigma^{(n)}=\sum_{i^{(n)}\in\mathbb{Z}/n\mathbb{Z}}\left\langle\mathcal{S}_{i^{(n)}},u\right\rangle+Q\sum_{i^{(n)}\in\mathbb{Z}/n\mathbb{Z}}d^\alpha(\mathcal{S}_{i^{(n)}}^R,\mathcal{S}_{i^{(n)}+1^{(n)}}^L).
\end{align*}
If there is some $j^{(n)}\in \mathbb{Z}/n\mathbb{Z}$ such that
$d(\mathcal{S}_{j^{(n)}}^R,\mathcal{S}_{j^{(n)}+1^{(n)}}^L)<\delta'$, just take
\[\mathcal{S}_{1^{(n-1)}}=\mathcal{S}_{j^{(n)}}*\mathcal{S}_{j^{(n)}+1^{(n)}}\text{ and }\mathcal{S}_{i^{(n-1)}}=\mathcal{S}_{j^{(n)}+i^{(n)}}\text{ for }i=2,3,\cdots n-1\]
$$\mathcal{L}^{(n-1)}=\{\mathcal{S}_{i^{(n-1)}},i^{(n-1)}\in\mathbb{Z}/(n-1)\mathbb{Z}\}$$
and
\begin{align*}
\Sigma^{(n-1)}=\sum_{i^{(n-1)}\in\mathbb{Z}/(n-1)\mathbb{Z}}\left\langle\mathcal{S}_{i^{(n-1)}},u\right\rangle+Q\sum_{i^{(n-1)}\in\mathbb{Z}/n\mathbb{Z}}d^\alpha(\mathcal{S}_{i^{(n-1)}}^R,\mathcal{S}_{i^{(n-1)}+1^{(n-1)}}^L).
\end{align*}
Note that by Lemma \ref{Bou-4}
\begin{align*}
&\Sigma^{(n)}-\Sigma^{(n-1)}\\
\ge&-\left| \left\langle\mathcal{S}_{j^{(n)}}*\mathcal{S}_{j^{(n)}+1^{(n)}},u\right\rangle-\left\langle\mathcal{S}_{j^{(n)}},u\right\rangle-\left\langle\mathcal{S}_{j^{(n)}+1^{(n)}},u\right\rangle\right|\\
&+Qd^\alpha\left(\mathcal{S}_{j^{(n)}}^R,\mathcal{S}_{(j+1)^{(n)}}^L\right)\\
&-Q\left(d^\alpha\left(\mathcal{S}_{j^{(n)}}^L,(\mathcal{S}_{j^{(n)}}*\mathcal{S}_{(j+1)^{(n)}})^L\right)-d^\alpha\left(\mathcal{S}_{(j+1)^{(n)}}^R,(\mathcal{S}_{j^{(n)}}*\mathcal{S}_{(j+1)^{(n)}})^R\right)\right)\\
&\ge 0.
\end{align*}
That is
\begin{align}\label{eq213}
\Sigma^{(n)}\ge\Sigma^{(n-1)}.
\end{align}
Repeat the above process until $\mathcal{L}^{(1)}$ with $d(\mathcal{S}_{1^{(1)}}^R,\mathcal{S}_{1^{(1)}}^L)<\delta'$ {\bf OR} some $m\in[1,n]\cap\mathbb{N}$ with \[d\left(\mathcal{S}_{j^{(m)}}^R,\mathcal{S}_{(j+1)^{(m)}}^L\right)\ge\delta'\text{ for all } j\in\mathbb{Z}/m\mathbb{Z}.\]
In the case that the process ends at $\mathcal{L}^{(1)}$ with $d(\mathcal{S}_{1^{(1)}}^R,\mathcal{S}_{1^{(1)}}^L)<\delta'$. We have by Lemma \ref{Bou-2} that
\begin{align}\label{eq-1}\begin{split}
\Sigma^{(1)}&=\left\langle\mathcal{S}_{1^{(1)}},u\right\rangle+Qd^\alpha(\mathcal{S}_{1^{(1)}}^R,\mathcal{S}_{1^{(1)}}^L)\\
&\ge -K_1 d^\alpha(\mathcal{S}_{1^{(1)}}^R,\mathcal{S}_{1^{(1)}}^L)+Qd^\alpha(\mathcal{S}_{1^{(1)}}^R,\mathcal{S}_{1^{(1)}}^L)\\
&\ge 0.
\end{split}
\end{align}
In the case that the process ends at some $m\in[1,n]\cap\mathbb{N}$ with \[d\left(\mathcal{S}_{j^{(m)}}^R,\mathcal{S}_{(j+1)^{(m)}}^L\right)\ge\delta'\text{ for all } j^{(m)}\in\mathbb{Z}/m\mathbb{Z}.\]
We have by Lemma \ref{Bou-3} that
\begin{align}\label{eq-2}\begin{split}
\Sigma^{(m)}&=\sum_{i^{(m)}\in\mathbb{Z}/m\mathbb{Z}}\left\langle\mathcal{S}_{i^{(m)}},u\right\rangle+Q\sum_{i^{(m)}\in\mathbb{Z}/m\mathbb{Z}}d^\alpha(\mathcal{S}_{i^{(m)}}^R,\mathcal{S}_{i^{(m)}+1^{(m)}}^L)\\
&\ge-mK_2\delta'^\alpha+mQ\delta'^\alpha\\
&\ge 0.
\end{split}
\end{align}
Combining the inequality \eqref{eq-1}, \eqref{eq-2} and the fact $\Sigma^{(n)}\ge \Sigma^{(n-1)}\ge \Sigma^{(n-2)}\ge\cdots$ by \eqref{eq213}, one has
\[\Sigma^{(n)}\ge0.\]
Then \[\sum_{i^{(n)}\in\mathbb{Z}/n\mathbb{Z}}u_\gamma(x_{i^{(n)}})+\frac{Q}{\gamma}\sum_{i^{(n)}\in\mathbb{Z}/n\mathbb{Z}}d^\alpha(\mathcal{S}^R_{i^{(n)}},\mathcal{S}^L_{i^{(n)}+1^{(n)}})=\frac{\Sigma^{(n)}}{\gamma}\ge 0.\]
By Lemma \ref{lemma-21}, there is an $\alpha$-H\"older
function $v$ on $\L$ with $\|v\|_\alpha \le \frac{Q}{\gamma}$ such that $$u_\gamma|_\L\ge v\circ \phi_\gamma|_\L-v.$$This ends the proof.
\end{proof}
Finally , we give the proof of Lemma \ref{revael}.
\begin{proof}[Proof of Lemma \ref{revael}]
(1). By Lemma \ref{MCGB}, we only need to show that
	\[\int u-\beta(u;\psi,\L,\Phi)\psi d\mu\ge 0\text{ for all } \mu\in\mathcal{M}(\Phi|_\L).\]
	It is immediately  from the fact
	\[\frac{\int ud\mu}{\int \psi d\mu}\ge\beta(u;\psi,\L,\Phi)\text{ for all } \mu\in\mathcal{M}(\Phi|_\L)\]
	since $\psi$ is strictly positive.
	
	(2).  Given a probability measure $\mu\in\mathcal{M}_{min}(u;\psi,\L,\Phi)$, one has
	\[\int u_\gamma+v\circ \phi_\gamma-v-\beta(u;\psi,\L,\Phi)\psi_\gamma d\mu=\int u-\beta(u;\psi,\L,\Phi)\psi d\mu=0.\]
	Combining (1) and the fact $u_\gamma|_\L+v\circ \phi_\gamma|_\L-v-\beta(u;\psi,\L,\Phi)\psi_\gamma|_\L$ is continuous on $\L$, one has
	\[supp(\mu)\subset \{x\in \L:(u_\gamma+v\circ \phi_\gamma-v-\beta(u;\psi,\L,\Phi)\psi_\gamma)|_\L(x)=0\}.\]
	Therefore,  \[Z_{u,\psi}\subset\left\{x\in \L:(u_\gamma|_\L+v\circ \phi_\gamma|_\L-v-\beta(u;\psi,\L,\Phi)\psi_\gamma|_\L)(x)=0\right\}.\] This ends the proof.
\end{proof}

\subsection{Proof of Lemma \ref{lemma-2}}\label{S:ProofPerApp}
In this section, we mainly prove the periodic approximation. The proof partially follows the arguments in \cite{BQ}.
\subsubsection{Joining of segments}
Recall that the definition of the joining of two segments $\mc S_1\star \mc S_2$  are given by (\ref{E:S1*S2}), we give some properties of jointed segments.
%
\begin{lem}\label{lem-Q-2}If two segments $\mathcal{S}_1$ and $\mathcal{S}_2$ satisfy $|\mathcal{S}_1|\ge 1$ and $d(\mathcal{S}_1^R,\mathcal{S}_2^L)\le \delta'$, then \begin{itemize}
		\item[(1).]$\max_{x\in \mathcal{S}_1*\mathcal{S}_2}d(x, \mathcal{S}_1\cup\mathcal{S}_2)\le C^3d(\mathcal{S}_1^R,\mathcal{S}_2^L);$
		\item[(2).] $|\mathcal{S}_1|+|\mathcal{S}_2|-1\le|\mathcal{S}_1*\mathcal{S}_2|\le |\mathcal{S}_1|+|\mathcal{S}_2|+1.$
	\end{itemize}
\end{lem}
\begin{proof}(1).  Note $v=v(\mathcal{S}_2^L,\mathcal{S}_1^R)$, $w=w(\mathcal{S}_2^L,\mathcal{S}_1^R)$ and
	$$\widetilde{\mathcal{S}}_2:[v,|\mathcal{S}_2|]:t\to\phi_t(\mathcal{S}_2^L).$$
	Then for $t\in[-|\mathcal{S}_1|,0]$,
	$$d(\phi_t(w),\mathcal{S}_1)\le d(\phi_t(w),\phi_t(\mathcal{S}_1^R))\le Ce^{\lambda t}d(w,\mathcal{S}_1^R)\le C^2e^{\lambda t}d(\mathcal{S}_1^R,\mathcal{S}_2^L)\le C^2d(\mathcal{S}_1^R,\mathcal{S}_2^L),$$
	where we used $w\in W_\epsilon^u(\mathcal{S}_1^R).$ For $t\in [0,|\mathcal{S}_2|-v]$,
	$$d(\phi_t(w),\widetilde{\mathcal{S}}_2)\le d(\phi_t(w),\phi_t(w)\phi_v(\mathcal{S}_2^L))\le Ce^{-\lambda t}d(w,\phi_v(\mathcal{S}_2^L))\le C^2e^{-\lambda t}d(\mathcal{S}_1^R,\mathcal{S}_2^L).$$
where we used $w\in W_\epsilon^s(\phi_v(\mathcal{S}_2^L)).$ Thus, for $t\in [0,|\mathcal{S}_2|-v]$,
	\begin{align*}
	d(\phi_t(w),\mathcal{S}_2)&\le d(\phi_t(w),\widetilde{\mathcal{S}}_2)+\max_{x\in \widetilde{\mathcal{S}}_2}d(\mathcal{S}_2,\widetilde{\mathcal{S}}_2)\\
	&\le C^2e^{-\lambda t}d(\mathcal{S}_1^R,\mathcal{S}_2^L)+d(\mathcal{S}_2^R,\phi_v(\mathcal{S}_2^R))\\
	&\le C^3d(\mathcal{S}_1^R,\mathcal{S}_2^L),
	\end{align*}
	where we used $C\gg1$ and 1b) of the {\bf basic cononical setting}.	Thus, by summing up,
	$$\max_{x\in \mathcal{S}_1*\mathcal{S}_2}d(x, \mathcal{S}_1\cup\mathcal{S}_2)=
	\max_{t\in[-|\mathcal{S}_1|,|\mathcal{S}_2|-v]} d(\phi_t(w), \mathcal{S}_1\cup\mathcal{S}_2)\le C^3d(\mathcal{S}_1^R,\mathcal{S}_2^L).$$
	
	(2). One has
	$$|\mathcal{S}_1*\mathcal{S}_2|=|\mathcal{S}_1|+|\mathcal{S}_2|-v\ge |\mathcal{S}_1|+|\mathcal{S}_2|-Cd(\mathcal{S}_1^R,\mathcal{S}_2^L)\ge |\mathcal{S}_1|+|\mathcal{S}_2|-C\delta'>|\mathcal{S}_1|+|\mathcal{S}_2|-1,$$
where we used the assumption $\delta'\ll\frac{1}{C}.$
	On the other hand, one also has that 
	$$|\mathcal{S}_1*\mathcal{S}_2|=|\mathcal{S}_1|+|\mathcal{S}_2|-v\le |\mathcal{S}_1|+|\mathcal{S}_2|+C\delta'\le |\mathcal{S}_1|+|\mathcal{S}_2|+1.$$
	This ends the proof.
\end{proof}

\begin{lem}\label{lem-Q-3}There exists a large constant $P_0>1$ such that if two segments $\mathcal{S}_1$ and $\mathcal{S}_2$ satisfy $|\mathcal{S}_1|\ge P_0$, $|\mathcal{S}_2|\ge P_0$ and $d(\mathcal{S}_1^R,\mathcal{S}_2^L)\le \delta'$, then
	$$d(\mathcal{S}_1^R,\mathcal{S}_2^L)\ge 2d(\mathcal{S}_1^L,(\mathcal{S}_1*\mathcal{S}_2)^L)+2d((\mathcal{S}_1*\mathcal{S}_2)^R,\mathcal{S}_2^R).$$
\end{lem}
\begin{proof} First we  fix a large constant $P_0\gg 1$ such that $C^2e^{-\lambda(P_0-1)}+C^2e^{-\lambda P_0}<\frac{1}{2}$. Fix two segments $\mathcal{S}_1$ and $\mathcal{S}_2$ as in Lemma.  Note $v=v(\mathcal{S}_2^L,\mathcal{S}_1^R)$ and  $w=w(\mathcal{S}_2^L,\mathcal{S}_1^R)$. Then
	\begin{align*}
	d(\mathcal{S}_1^L,(\mathcal{S}_1*\mathcal{S}_2)^L)&= d(\phi_{-|\mathcal{S}_1|}(\mathcal{S}_1^R),\phi_{-|\mathcal{S}_1|}(w))\\
	&\le Ce^{-\lambda|\mathcal{S}_1|}d(\mathcal{S}_1^R,w)\\
	&\le C^2e^{-\lambda P_0}d(\mathcal{S}_1^R,\mathcal{S}_2^L),
	\end{align*}
	where we used $w\in W_\epsilon^u(\mathcal{S}_1^R).$
On the other hand,
	\begin{align*}
	d(\mathcal{S}_2^R,(\mathcal{S}_1*\mathcal{S}_2)^R)&= d(\phi_{|\mathcal{S}_2|-v}\phi_v(\mathcal{S}_2^L),\phi_{|\mathcal{S}_2|-v}(w))\\
	&\le Ce^{-\lambda(|\mathcal{S}_2|-v)}d(\phi_v(\mathcal{S}_2^L),w)\\
	&\le C^2e^{-\lambda (P_0-1)}d(\mathcal{S}_1^R,\mathcal{S}_2^L),
	\end{align*}
	where we used $w\in W_\epsilon^u(\phi_v(\mathcal{S}_2^L)).$
By assumption, we have
	\begin{align*}	d(\mathcal{S}_1^L,(\mathcal{S}_1*\mathcal{S}_2)^L)+d(\mathcal{S}_2^R,(\mathcal{S}_1*\mathcal{S}_2)^R)&\le C^2e^{-\lambda(P_0-1)}d(\mathcal{S}_1^R,\mathcal{S}_1^L)+C^2e^{-\lambda P_0}d(\mathcal{S}_1^R,\mathcal{S}_1^L)\\
	&\le \frac{1}{2}d(\mathcal{S}_1^R,\mathcal{S}_1^L).
	\end{align*}
	This ends the proof.
\end{proof}
\subsubsection{Periodic approximation}
For integer $n\ge 1$, let $\Sigma_n=\{0,1,2,\cdot,n-1\}^{\mathbb{N}}$ and $\sigma$ be a shift on $\Sigma_n$. Assume $F$ is a subset of $\bigcup_{i\ge 1}\{0,1,2,\cdot,n-1\}^{i}$, then the subshift with forbidden $F$ is denoted by $(Y_F,\sigma)$ where
$$Y_F=\left\{x\in \{0,1,2,\cdot,n-1\}^{\mathbb{N}},w\text{ does not appear in } x\text{ for all }w\in F\right\}.$$ The following lemma is Lemma 5 of \cite{BQ}, which will be used later.
\begin{lem}[\cite{BQ}]\label{lemma-1} Suppose that $(Y,\sigma)$ is a shift of  finite type $($with forbidden words of length 2$)$ with $M$ symbols and entropy $h$. Then $(Y,\sigma)$ contains a periodic point of period at most $1+Me^{(1-h)}$.
\end{lem}
%

Now we are ready to prove Lemma \ref{lemma-2}, which partially follow the argument in \cite{BQ}.
\begin{proof}[ Proof of Lemma \ref{lemma-2}]	

	  Fix a positive constant $\delta''\ll \frac{\delta'}{C^{10}e^\beta}$.  Let  $\mathcal{P}=\{B_1,B_2,\cdots,B_m\}$ be a finite partition of $\L$ with diameter smaller than $\delta''$. For $x\in \L$, $\widehat{ x}\in\{1,2,3,\cdots,m\}^{\mathbb{N}}$ is defined by
	\[\widehat{x}(n)=j\text{ if } \phi_n(x)\in B_j, n=0,1,2,\cdots.\]
	Denote $\widehat{Z}=\{\widehat{x}: x\in Z\}$ and $W_n$ the collection of
	length $n$ string that appears in $\widehat{Z}$. One has
	\[\sharp W_n=K_ne^{nh}\]where $h=h_{top}(\overline{\widehat{Z}},\sigma)$ and $K_n$ grows at a subexponential rate. Let
	\[Y_n=\{y_0y_1y_2\cdots\in W_n^\mathbb{N}:y_iy_{i+1}\in W_{2n}\text{ for all }i\in\mathbb{N}\}\]
	and $(Y_n,\sigma_n)$ is the 1-step shift of finite type on $W_n$.
	From Lemma \ref{lemma-1}, the shortest periodic orbit in $Y_n$ is at most $1+K_ne^{nh}e^{1-nh}=1+eK_n$.  Denote one of the shortest periodic orbit in $Y_n$ by $z_0z_1z_2\cdots z_{p_n-1}z_0z_1z_2\cdots$ for some $p_n\le 1+eK_n$ and $z_i\in W_n, i=0,1,2,\cdots,p_n-1$. For $i=0,1,2,\cdots, p_n-1$, there is $x_i\in Z$ such that the leading $2n$ string of $\widehat{x_i}$ is $z_iz_{i+1}$ (we note $z_{p_n}=z_{0}$, $x_{p_n}=x_{0}$, $\mathcal{S}_{p_n}=\mathcal{S}_{0},\cdots$). Choose segments $\mathcal{S}_i$ by
	\[\mathcal{S}_i:[\frac{n}{2},\frac{3n}{2}+v_i]\to M:t\to \phi_t(x_i)\text{ for }i=0,1,2,\cdots,p_n-1,\]
	where $v_i=v(\phi_n(x_i),x_{i+1})$. We have the following Claim.
	
	\noindent{\it {\bf Claim Q1. }  $d(\mathcal{S}_i^R,\mathcal{S}_{i+1}^L)\le 2C^2e^{-\frac{n\lambda}{2}}\delta''$ for $i=0,1,2,\cdots,p_n-1$.}
\begin{proof}[Proof of Claim Q1.] Note that the leading $2n$ string of $\widehat{x_i}$ is $z_iz_{i+1}$ and leading $n$ string of $\widehat{x_{i+1}}$ is $z_{i+1}$, which means $$\mathcal{P}(\phi_{n+j}(x_i))=\mathcal{P}(\phi_{j}(x_{i+1}))\text{ for }j=0,1,\cdots,n-1.$$
	Therefore, $$d(\phi_{n+j}(x_i),\phi_{j}(x_{i+1}))<\delta''\text{ for }j=0,1,\cdots,n-1.$$ Thus
	$$d(\phi_{n+t}(x_{i}),\phi_{t}(x_{i+1}))<Ce^\beta\delta''<\delta'\text{ for }t\in[0,n].$$
	Then by Lemma \ref{An-1}, we have
	\begin{align*}d\left(\phi_{\frac{3n}{2}+v_i}(x_{i}),\phi_{\frac{n}{2}}(x_{i+1})\right)&\le C^2e^{-\frac{n\lambda}{2}}(d(\phi_n(x_i),x_{i+1})+d(\phi_{2n}(x_i),\phi_{n}(x_{i+1})))\\
	&\le 2C^2e^{-\frac{n\lambda}{2}}\delta''.
	\end{align*}
	This ends the proof of { Claim Q1}.		
\end{proof}
	
Now we define segments $\overline{\mathcal{S}}_i$ recursively
for $0\le i\le p_n-1$ by $\overline{\mathcal{S}}_0:=\mathcal{S}_0$ and
$$\overline{\mathcal{S}}_i:=\overline{\mathcal{S}}_{i-1}*\mathcal{S}_i\text{ for }1\le i\le p_n-1.$$
Based on Claim Q1, we have the following claim.
	
\noindent{\it {\bf Claim Q2. }There is a positive integer $N$ such that for any $n\ge N$, one has
\begin{itemize}
			\item[(1).]   $\overline{\mathcal{S}}_i$ is well defined for $0\le i\le p_n-1;$
			\item[(2).]   $d(\overline{\mathcal{S}}_{i}^R, \mathcal{S}_{i+1}^L)\le 4C^2p_ne^{-\frac{n\lambda}{2}}\delta''<\delta'$ for $0\le i\le p_n-2;$
			\item[(3).]  $ d(\overline{\mathcal{S}}_{p_n-1}^R, \overline{\mathcal{S}}_{p_n-1}^L)\le 2C^2p_ne^{-\frac{n\lambda}{2}}\delta''<\min\left\{\delta',\frac1L\right\}$, where $L$ is as in Lemma \ref{lemma-a};
			\item[(4).] $(n-1)p_n\le |\overline{\mathcal{S}}_{p_n-1}|\le (n+1)p_n;$
			\item[(5).]   $\max_{x\in \overline{\mathcal{S}}_{p_n-1}}d(x, Z)\le C^4p_n^2e^{-\frac{n\lambda}{2}}\delta''.$
\end{itemize}}
	\begin{proof}[Proof of Claim Q2.] Since $p_n$ grows at a subexponential rate, we can take $N$ large enough such that
	\begin{align}\label{Q2-1}
	N>P_0\text{ and }4p_nC^2e^{-\frac{n\lambda}{2}}\delta''< \min\left\{\delta',\frac1L\right\}\text{ for all }n\ge N,
	\end{align}
	where $P_0$ is the constant as in Lemma \ref{lem-Q-3}.
	For $0\le i\le p_n-2$, we define
	$$\chi(i)=d(\overline{\mathcal{S}}_i^R,\mathcal{S}_{i+1}^L)+d(\mathcal{S}_{i+1}^R,\mathcal{S}_{i+2}^L)+\cdots+d(\mathcal{S}_{p_n-2}^R,\mathcal{S}_{p_n-1}^L)+d(\mathcal{S}_{p_n-1}^R,\overline{\mathcal{S}}_i^L).$$
	By Claim Q1,
	$$\chi(0)\le 2C^2p_ne^{-\frac{n\lambda}{2}}\delta''.$$	
	Now we are to show that $\chi(i)$ and  $\overline{\mathcal{S}}_i$ are well defined,  which satisfy that
	$$\chi(i)\le \delta'\text{ and }|\overline{\mathcal{S}}_i|>P_0\text{ for }i=0, 1, 2, \cdots,p_n-2.$$ These are clearly true for $i=0$. Now we assume that these are true for some $i\in\{0,1,2\cdots,p_n-2\}$. Then for $i+1$, since $\chi(i)\le \delta'$, one has $d(\overline{\mathcal{S}}_i^R,\mathcal{S}_{i+1}^L)\le \delta'$.
	Thus, we can join $\overline{\mathcal{S}}_i$ and $\mathcal{S}_{i+1}$ by letting
	$$\overline{\mathcal{S}}_{i+1}=\overline{\mathcal{S}}_i*\mathcal{S}_{i+1}.$$
	It is clear that $|\overline{\mathcal{S}}_{i+1}|>P_0$ by Lemma \ref{lem-Q-2} (2).
	
	On the other hand, by triangle inequality and Lemma \ref{lem-Q-3},  one has that
	\begin{align}\begin{split}\label{Q-315}
	&\chi(i)-\chi(i+1)\\
	=&d(\overline{\mathcal{S}}_i^R,\mathcal{S}_{i+1}^L)+d(\mc S_{i+1}^R,\mc S_{i+2}^L)-d(\overline{\mathcal{S}}_{i+1}^R,\mathcal{S}_{i+2}^L)+d(\mathcal{S}_{p_n-1}^R,\overline{\mathcal{S}}_i^L)-d(\mathcal{S}_{p_n-1}^R,\overline{\mathcal{S}}_{i+1}^L)\\
	\ge& d(\overline{\mathcal{S}}_i^R,\mathcal{S}_{i+1}^L)- d(\overline{\mathcal{S}}_i^L,\overline{\mathcal{S}}_{i+1}^L)-d(\overline{\mathcal{S}}_{i+1}^R,\mathcal{S}_{i+1}^R)\\
	\ge&\frac12 d(\overline{\mathcal{S}}_i^R,\mathcal{S}_{i+1}^L).
	\end{split}\end{align}
	Therefore, $\chi(i+1)\le \chi(i)\le \delta'$. By induction, we can finish the proof of (1).

	By \eqref{Q-315}, one has
	$$d(\overline{\mathcal{S}}_{i}^R, \mathcal{S}_{i+1}^L)\le2 \chi(i)\le  \cdots\le 2\chi(0)\le 4C^2p_ne^{-\frac{n\lambda}{2}}\delta''\text{ for }i=0,1,2\cdots, p_n-2,$$
	and we can deduce the following by triangle inequality:
	$$d(\overline{\mathcal{S}}_{p_n-1}^R, \overline{\mathcal{S}}_{p_n-1}^L)\le\chi(p_n-2)\le \cdots\chi(0)\le 2C^2p_ne^{-\frac{n\lambda}{2}}\delta''.$$
	This ends the proof of (2) as well as (1), (3), and (4) follows from Lemma \ref{lem-Q-2} (2).
	
	Now we denote $D_i:=\overline{\mathcal{S}}_i\bigcup\cup_{j=i+1}^{p_n-1}\mathcal{S}_j$ for $i=0,1,2,\cdots,p_n-1.$ Then by Lemma \ref{lem-Q-2} (1),
	$$\max_{x\in D_{i+1}}d(x,D_i)\le \max_{x\in \overline{\mathcal{S}}_i*\mathcal{S}_{i+1}}d(x, \overline{\mathcal{S}}_i\cup\mathcal{S}_{i+1})\le C^3d(\mathcal{S}_i^R,\mathcal{S}_{i+1}^L)\le 2C^5p_ne^{-\frac{n\lambda}{2}}\delta''.$$
	Therefore, by triangle inequality and the fact that $D_0\subset Z$,
	$$\max_{x\in \overline{\mathcal{S}}_{p_n-1}}d(x, Z)\le \max_{x\in D_0}d(x, Z)+\sum_{i=0}^{p_n-2}\max_{x\in D_{i+1}}d(x,D_i)\le 2C^5p_n^2e^{-\frac{n\lambda}{2}}\delta''.$$
	This ends the proof of Claim Q2.
\end{proof}

 Recall that  $P_0$ is the constant as in Lemma \ref{lem-Q-3}, $K, L$ are the constants as in Lemma \ref{lemma-a} and $N$ is the constant  as in Claim Q2. We fix an integer $n>\max(P_0, K, N)+1$ and let $\overline{\mathcal{S}}_{p_n-1}$ be the segment as in Claim Q2. Then by Claim Q2 (3),
$$ |\overline{\mathcal{S}}_{p_n-1}|>K \text{ and } d(\overline{\mathcal{S}}_{p_n-1}^R, \overline{\mathcal{S}}_{p_n-1}^L)\le\delta'.$$
Applying the  Anosov Closing Lemma, we have a periodic segment $\mathcal{O}_n$ such that \begin{align}\label{e-Q-1}
	\left||\overline{\mathcal{S}}_{p_n-1}|-|\mathcal{O}_n|\right|\le Ld\left(\overline{\mathcal{S}}_{p_n-1}^L,\overline{\mathcal{S}}_{p_n-1}^R\right)
	\end{align} and
\begin{align}\label{e-Q-2} d\left(\phi_t(\mathcal{O}_n^L),\phi_t(\overline{\mathcal{S}}_{p_n-1}^L)\right)\le Ld(\overline{\mathcal{S}}_{p_n-1}^L, \overline{\mathcal{S}}_{p_n-1}^R)\ \forall t\in \left[0, \max\left(|\overline{\mathcal{S}}_{p_n-1}|,|\mathcal{O}_n|\right)\right].
\end{align}

We claim the following:

{\it {\bf Claim Q3. }$\max_{x\in \mathcal{O}_n}d(x, Z)\le (2C^3Lp_n+C^4p_n^2)e^{-\frac{n\lambda}{2}}\delta''.$}
\begin{proof}[Proof of Claim Q3]	
	If $|\mathcal{O}|\le |\overline{\mathcal{S}}_{p_n-1}|$, by \eqref{e-Q-2},
	$$\max_{x\in \mathcal{O}}d(x, \overline{\mathcal{S}}_{p_n-1})\le Ld\left(\overline{\mathcal{S}}_{p_n-1}^L,\overline{\mathcal{S}}_{p_n-1}^R\right).$$
	If $|\mathcal{O}|> |\overline{\mathcal{S}}_{p_n-1}|$, note $t_*=\min(t, |\overline{\mathcal{S}}_{p_n-1}|)$ for $t\in [0,|\mathcal{O}|]$. Then,  by (\ref{e-Q-1}) and (\ref{e-Q-2}), one has that for $t\in [0,|\mathcal{O}|]$
	\begin{align*}
	d\left(\phi_t(\mathcal{O}^L), \overline{\mathcal{S}}_{p_n-1}\right)&\le d\left(\phi_{t}(\mathcal{O}^L), \phi_{t_*}(\mathcal{O}^L)\right)+d\left(\phi_{t_*}(\mathcal{O}^L),\phi_{t_*}(\overline{\mathcal{S}}_{p_n-1}^L)\right)\\
	&\le CLd\left(\overline{\mathcal{S}}_{p_n-1}^L,\overline{\mathcal{S}}_{p_n-1}^R\right)+Ld\left(\overline{\mathcal{S}}_{p_n-1}^L,\overline{\mathcal{S}}_{p_n-1}^R\right)\\
	&\le C^2Ld\left(\overline{\mathcal{S}}_{p_n-1}^L,\overline{\mathcal{S}}_{p_n-1}^R\right),
	\end{align*}
	where $C\gg 1$ as in Remark \ref{remark-1-1}.
	Therefore, $$\max_{x\in \mathcal{O}}d\left(x, \overline{\mathcal{S}}_{p_n-1}\right)=\max_{t\in [0,|\mathcal{O}|]}d\left(\phi_t(\mathcal{O}^L), \overline{\mathcal{S}}_{p_n-1}\right)\le C^2Ld\left(\overline{\mathcal{S}}_{p_n-1}^L,\overline{\mathcal{S}}_{p_n-1}^R\right).$$
	Combining with (3), (5) of Claim Q2, we have
	\begin{align*}\max_{x\in \mathcal{O}}d(x, Z)&\le \max_{x\in \mathcal{O}}d\left(x, \overline{\mathcal{S}}_{p_n-1}\right)+\max_{x\in \overline{\mathcal{S}}_{p_n-1}}d(x, Z)\\
	&\le  C^2L\cdot 2C^2p_ne^{-\frac{n\lambda}{2}}\delta''+C^5p_n^2e^{-\frac{n\lambda}{2}}\delta''\\
	&= (2C^4Lp_n+C^5p_n^2)e^{-\frac{n\lambda}{2}}\delta''.
	\end{align*} This ends the proof of Claim Q3.
\end{proof}

By {Claim Q3} and (\ref{e-Q-1}), we have
\begin{align*}
d_{\alpha,Z}(\mathcal{O}_n)&\le |\mathcal{O}_n|\left(\max_{x\in \mathcal{O}}d(x, Z)\right)^\alpha\\
&\le \left(|\overline{\mathcal{S}}_{p_n-1}|+ Ld\left(\overline{\mathcal{S}}_{p_n-1}^L,\overline{\mathcal{S}}_{p_n-1}^R\right)\right)\cdot\left((2C^4Lp_n+C^5p_n^2)e^{-\frac{n\lambda}{2}}\delta''\right)^\alpha\\
&\le H_ne^{-\frac{n\alpha\lambda}{2}},
\end{align*}	
where $H_n=	\big((2C^4Lp_n+C^5p_n^2)\delta''\big)^\alpha\cdot((n+1)p_n+ 1)$. Note that $H_n$ grows at a subexponential rate as $n$ increases as  $p_n$ does. Hence
	\[\limsup_{P\to+\infty}P^k\min_{\mathcal{O}\in \mathcal{O}^P_\L}d_{\alpha,Z}(\mathcal{O})\le \limsup_{n\to+\infty}((n+1)p_n+1)^k\cdot H_ne^{-\frac{n\alpha\lambda}{2}}=0,\]
	where we used the fact that $p_n, H_n$ grow at a subexponential rate as $n$ increases and $|\mathcal{O}_n|\le np_n+1$.
	The proof of Lemma \ref{lemma-2} is completed.
\end{proof}

\section{Further discussions on the case of $C^{s,\a}$-observables}\label{S:HighReg}
For $s\in\mathbb{N}$, $0\le\alpha\le 1$ and a strictly positive function $\psi$ on $M$,  $Per^*_{s,\alpha}(M,\psi)$ is defined as the collection of $C^{s,\alpha}$-continuous functions on $M$, such that for each $u\in Per^*_{s,\alpha}(M,\psi)$, $\mathcal{M}_{min}(u;\psi,\L,\Phi)$ contains at least one periodic measure. And $Loc_{s,\alpha}(M,\psi)$ is defined by
\begin{align*}Loc_{s,\alpha}(M,\psi):=\{&u\in Per^{*}_{s,\alpha}(M,\psi):\text{ there is } \varepsilon>0\text{ such that }\\
&\mathcal{M}_{min}(u+h;\psi,\L,\Phi)=\mathcal{M}_{min}(u;\psi,\L,\Phi)\text{ for all }\|h\|_{r,\alpha}<\varepsilon \}.\end{align*}
In the case $s\ge1$ and $\alpha>0$ or $s\ge 2$, we do not have analogous result like Proposition \ref{prop-2}. However, we have the following weak version.
\begin{prop}\label{prop-4}
	 Let $\mathcal{O}$ be a periodic segment of $\Phi|_\L$ with $D(\mathcal{O})>0$ and $u\in \mathcal{C}(M)$ with $u|_\mathcal{O}=0$ and $u|_{M/\mathcal{O}}>0$. Then there exists a constant $\varrho>0$  such that the probability measure  $$\mu_{\mathcal{O}}\in\mathcal{M}_{min}(u+h;\psi,\L,\Phi),$$ where $h$ is any $\mathcal{C}^{0,1}(M)$ function with $\|h\|_1<\varrho$.
\end{prop}
\begin{rem}\label{remark-3} As in Remark \ref{remark-2}, for $s\in\mathbb{N}$ and $0\le \alpha\le 1$, we let $\widetilde w\in\mathcal{C}^{s,\a}(M)$ such  that $\|\widetilde{w}\|_{s,\alpha}<\varepsilon$, $\widetilde{w}|_\mathcal{O}=0$ and $\widetilde w|_{M\setminus\mathcal{O}}>0$.  Then $\mu_\mathcal{O}$ is the unique measure in $\mathcal{M}_{min}(u+\widetilde w+h;\psi,\L,\Phi)$  whenever $\|h\|_{s,\alpha}<\varrho$. The Proposition shows that there is an open subset of $\mathcal{C}^{s,\alpha}({M})$ near $u$ such that functions in the open set have the same unique minimizing measure with respect to $\psi$ and the probability measure supports on a periodic orbit.
\end{rem}
By using Remark \ref{remark-3}, we have the following result.
\begin{thm}
$Loc_{s,\alpha}(M,\psi)$ is an open dense subset of $Per^*_{s,\alpha}(M,\psi)$ w.r.t. $\|\cdot\|_{s,\alpha}$ for integer $s\ge 1$ and real number $0\le\alpha\le 1$.
\end{thm}
\begin{proof}Given $s\ge 1$ and $0\le \alpha\le 1$. The openness is clearly true. We prove $Loc_{s,\alpha}(M,\psi)$ is dense in $Per^*_{s,\alpha}(M,\psi)$ w.r.t. $\|\cdot\|_{s,\alpha}$. Since
	$$\int ud\mu=\int \bar ud\mu\text{ for all }\mu\in\mathcal{M}(\Phi|_\L),$$
we have $\mathcal{M}_{min}(u;\psi,\L,\Phi)=\mathcal{M}_{min}(\bar u;\psi,\L,\Phi)$. Then the theorem follows from Remark \ref{remark-3} immediately.
\end{proof}
\subsection{Proof of Proposition \ref{prop-4}}
Now we finish the proof of Proposition \ref{prop-4}.
\begin{proof}[Proof of Proposition \ref{prop-4}]	
 Let $\mathcal{O}$ be a periodic segment of $\Phi|_\L$ and $u\in \mathcal{C}(M)$ with $u|_\mathcal{O}=0$ and $u|_{M/\mathcal{O}}>0$.	For $0\le \rho\le D(\mathcal{O})$, we note $\theta( \rho)=\min\{u(x):d(x,\mathcal{O})\ge  \rho, x\in M\}.$ It is clear that  $\theta(0)=0$, $\theta(\rho)>0$ for $ \rho\neq 0$ and $\theta$ is non-decreasing. Since $D(\mathcal{O})>0$ by assumption,	there are two constants $\rho_1,\rho_2$ satisfy
\begin{align}\label{bbb-0}
0<\rho_1<\rho_2<\frac{D(\mathcal{O})}{4C^3e^{2\beta}}.	\end{align}
 Next we will show that $\mu_{\mathcal{O}}\in\mathcal{M}_{min}(u+h;\psi, \Lambda,\Phi)$ for all $h\in C^{0,1}(M)$ with $\|h\|_1<\varrho$,
	where the constant $\varrho$ is positive and
\begin{align}\label{b-0}
\varrho<\frac{1}{2}\min\left\{\frac{\psi_{min}\theta(\rho_1)}{1+\psi_{min}},\frac{\theta(\frac{D(\mathcal{O})}{4C^2e^{2\beta}})}{(1+\frac{\|\psi\|_1}{\psi_{min}}) \cdot \frac{4C^{3}\rho_2}{\lambda}+|\mathcal{O}|\cdot  \frac{1+\psi_{min}}{\psi_{min}}+\frac{1+\psi_{min}}{\psi_{min}}}\right\}.	\end{align}
	Now we fix a function $h$ as above.	Note $G=u+h-a_\mathcal{O}\psi$ where
	\begin{align}\label{b-1}
	\begin{split}
	a_{\mathcal{O}}&=\frac{\left\langle \mathcal{O},u+h\right\rangle }{\left\langle \mathcal{O},\psi\right\rangle}\le\frac{\|h\|_0}{\psi_{min}}.
	\end{split}
	\end{align}
	Then $\frac{\int Gd\mu}{\int \psi d\mu}=\frac{\int u+hd\mu}{\int \psi d\mu}-a_\mathcal{O}.$
	Therefore, to show that $\mu_{\mathcal{O}}\in\mathcal{M}_{min}(u+h;\psi,\L,\Phi)$, it is enough to show that
	\[\int Gd\mu\ge0\text{ for all }\mu\in\mathcal{M}^e(\Phi|_\L),\]
	where we used the assumption $\psi$ is strictly positive and the fact $\int Gd\mu_\mathcal{O}=0$. Now we let
	$Area_1:=\{y\in M:d(y,\mathcal{O})\le \rho_1\}.$ We have the following claim.
	
\noindent{\it {\bf Claim F1.} $Area_1$ contains all $x\in M$ with $G(x)\le0$. }
\begin{proof}[Proof of Claim F1]
		For $x\notin Area_1$, we have
		\begin{align*}G(x)&= u(x)+h(x)-a_{\mathcal{O}}\psi\ge \theta(\rho_1)-\|h\|_0-a_{\mathcal{O}}\|\psi\|_0\ge\theta(\rho_1)-\frac{1+\psi_{min}}{\psi_{min}}\|h\|_0>0.	
		\end{align*}
		where we used \eqref{b-0} and \eqref{b-1}. This ends the proof of Claim F1.\end{proof}	
	Note $Area_2=\{y\in M:d(y,\mathcal{O})\le \rho_2\}.$ It is clear that $Area_1$ is in the interior of $Area_2$. Thus, $d(Area_1,M\setminus Area_2)>0$. Therefore, by Claim F1, we can fix a constant $0<\tau<1$ such that $G(\phi_t(x))>0$ for all $x\in M\setminus Area_2$ and $|t|\le\tau$.
	
\noindent{\it {\bf Claim F2.} If $z\in \L$ is not a generic point of $\mu_\mathcal{O}$, then there is $m\ge\tau$ such that $\int_{0}^mG(\phi_t(z))dt>0$.}
	
	Next we prove Proposition \ref{prop-4} by assuming the validity of  {Claim F2}, proof of which is left to the next subsection. Same as  the argument at the beginning of the proof, it is enough to show that for all $\mu\in\mathcal{M}^e(\Phi|_\L)$
	\[\int Gd\mu\ge0.\]
	Given $\mu\in\mathcal{M}^e(\Phi|_\L)$, in the case $\mu=\mu_{\mathcal{O}}$, it is obviously true.  In the case $\mu\neq\mu_{\mathcal{O}}$, just let $z$ be a generic point of $\mu$. Note that  $z$ is not a generic point of $\mu_{\mathcal{O}}$. By Claim F2, we have $t_1\ge\tau$ such that \[\int_0^{t_1} G(\phi_t(z))dt>0.\]
	Note that  $\phi_{t_1}(z)$ is still not a generic point of $\mu_{\mathcal{O}}$. Apply Claim F2 again, we have $t_2\ge t_1+\tau$ such that \[\int_{t_1}^{t_2} G(\phi_t(z))dt>0.\]
	By repeating the above process, we have $0\le t_1<t_2<t_3<\cdots$ with the gap not less than $\tau$ such that
	\[\int_{t_i}^{t_{i+1}} G(\phi_t(z))dt>0\text{ for } i=0,1,2,3,\cdots,\]
	where $t_0=0$. Therefore,
	\begin{align*}\int Gd\mu&=\lim_{m\to+\infty}\frac{1}{m}\int_{0}^{m}G(\phi_t(z))dt\\
	&= \lim_{i\to+\infty}\frac{1}{t_i}\left(\int_{t_0}^{t_1}G(\phi_t(z))dt+\int_{t_1}^{t_2}G(\phi_t(z))dt+\cdots+\int_{t_{i-1}}^{t_i}G(\phi_t(z))dt\right)\\
	&\ge 0.
	\end{align*}
	That is, $\mu_{\mathcal{O}}\in\mathcal{M}_{min}(u+h;\psi,\L,\Phi)$. This ends the proof of the Proposition.
\end{proof}

\subsection{Proof of Claim F2}
Assume that $z$ is not a generic point of $\mu_\mathcal{O}$, if $z\notin Area_2$, let $m=\tau$,
	we have nothing to prove  since $G(\phi_t(z))>0$ for all $|t|\le\tau$.
Now we assume that $z\in Area_2$. In the case that $d(\phi_t (z),\mathcal{O})<\frac{D(\mathcal{O})}{4C^2e^\beta}$ for all $t\ge 0$, by Lemma \ref{M-3}, $z$ is a generic point of $\mu_\mathcal{O}$, which contradicts to our assumption.
	Hence, there must be some $m_1>0$ such that $$d(\phi_{m_1} (z),\mathcal{O})\ge\frac{D(\mathcal{O})}{4C^2e^\beta}.$$
	We can assume $m_2>0$ be the smallest time such that
	\begin{align}\label{b-2}
	d(\phi_{m_2} (z),\mathcal{O})\ge\frac{D(\mathcal{O})}{4C^2e^\beta}.
	\end{align}
	The existence of $m_2$ is ensured by \eqref{bbb-0}. Then for $0\le t\le 1$,
	\[d(\phi_{m_2-t}(z),\mathcal{O})\ge\frac{D(\mathcal{O})}{4C^3e^{2\beta}}.\]	
	Then
	
	\begin{align}\label{b-12}
	\begin{split}\int_{m_2-1}^{m_2}G(\phi_{t}(z))dt&=\int_{m_2-1}^{m_2} {u}(\phi_{t}(z))+h(\phi_{t}(z))-a_{\mathcal{O}}\psi(\phi_{t}(z))dt\\
	&\ge\int_{m_2-1}^{m_2} \theta\left(\frac{D(\mathcal{O})}{4C^3e^{2\beta}}\right)-\|h\|_0- a_{\mathcal{O}}\|\psi\|_0dt\\
	&\ge\theta\left(\frac{D(\mathcal{O})}{4C^3e^{2\beta}}\right)- \frac{1+\psi_{min}}{\psi_{min}}\|h\|_0.\\
	\end{split}
	\end{align}
	where we used \eqref{b-1} and the definition of $\theta(\cdot)$.
	On the other hand, one has that
	$\frac{D(\mathcal{O})}{4C^3e^{2\beta}}>\rho_2$, which implies that
	\begin{align}\label{b-4}
	\phi_{m_2-t}(z)\notin Area_2\text{ for all }0\le t\le 1.
	\end{align}
	Since $Area_2$ is compact, we can take  $m_3$ the largest time with $0\le m_3\le m_2$ such that
	\[\phi_{m_3}(z)\in Area_2,\]
	where we use the assumption $z\in Area_2$.
	By \eqref{b-4}, it is clear that $m_3<m_2-1$.
	Then by Claim F1 and the fact that $Area_1\subset Area_2$,  \begin{align}\label{b-5}
	G(\phi_t(z))>0\text{  for all }m_3<t<m_2-1.\end{align}
	Since  $m_3<m_2$, one has by \eqref{b-2} that
	$$d(\phi_t(z),\mathcal{O})<\frac{D(\mathcal{O})}{4C^2e^\beta}<\delta'\text{ for all } 0\le t\le m_3.$$ Therefore, by Lemma \ref{M-3}, there is $y_0\in\mathcal{O}$ such that
	$$d(\phi_t(z), \phi_t(y_0))\le Cd(\phi_t(z),\mathcal{O})\le \frac{D(\mathcal{O})}{4Ce^\beta}\text{ for all }t\in[0,m_3].$$
	Also notice that
	$$d(z,y_0)\le C\rho_2\text{ and }d(\phi_{m_3}(z),\phi_{m_3}(y_0))\le C\rho_2,$$
	where we used $z,\phi_{m_3}(z)\in Area_2$.
	By using Lemma \ref{An-1}, we have for all $0\le t\le m_3$,
	\begin{align*}
	d(\phi_t(z),\phi_t\phi_v(y_0))&\le C^2e^{-\lambda\min(t,m_3-t)}(d(z,y_0)+d(\phi_{m_3}(z),\phi_{m_3}(y_0)))\\
	&\le 2C^3e^{-\lambda\min(t,m_3-t)}\rho_2,
	\end{align*}
	where $v=v(y_0,z)$.
	Hence,
	\begin{align*}
	\begin{split}\int_{0}^{ m_3}d(\phi_{t}(z),\phi_{t}\phi_v(y_0))dt&\le\int_{0}^{ m_3}2C^3\rho_2(e^{-\lambda t}+e^{-\lambda(m_3-t)}) dt\le\frac{4C^{3}\rho_2}{\lambda}.
	\end{split}
	\end{align*}
	Since $u(\phi_{t}(y_0))=0$ for all $t\in\mathbb{R}$ and $u\ge 0$, one has
	\begin{align}\label{b-6}
\begin{split}&\int_{0}^{ m_3}G(\phi_{t}(z))-G(\phi_{t+v}(y_0))dt\\
=&\int_{0}^{ m_3}{u}(\phi_{t}(z))+h(\phi_{t}(z))-{u}(\phi_{t+v}(y_0))-h(\phi_{t+v}(y_0))-a_\mathcal{O}(\psi(\phi_t(z))-\psi(\phi_{t+v}(y_0)))dt\\
\ge&\int_{0}^{ m_3}h(\phi_{t}(z))-h(\phi_{t+v}(y_0))-a_\mathcal{O}(\psi(\phi_t(z))-\psi(\phi_{t+v}(y_0)))dt\\
\ge& -(\|h\|_1+|a_\mathcal{O}|\|\psi\|_1)\int_{0}^{ m_3}d(\phi_{t}(z),\phi_{t+v}(y_0))dt\\
\ge& - (\|h\|_1+\frac{\|h\|_0\|\psi\|_1}{\psi_{min}}) \cdot \frac{4C^{3}\rho_2}{\lambda}\\
\ge& - \|h\|_1(1+\frac{\|\psi\|_1}{\psi_{min}}) \cdot \frac{4C^{3}\rho_2}{\lambda}.
\end{split}
\end{align}
	By assuming that $m_3=p|\mathcal{O}|+q$ for some nonnegative integer $p$ and real number $0\le q\le |\mathcal{O}|$, one has by \eqref{b-2} that
	\begin{align}\label{b-7}
	\begin{split}
	\int_{0}^{ m_3}G(\phi_{t+v}(y_0))dt&=\int_{m_3-q}^{m_3}G(\phi_{t+v}(y_0))dt\ge -|\mathcal{O}|\cdot \frac{1+\psi_{min}}{\psi_{min}}\|h\|_0,
	\end{split}
	\end{align}
	where we used $\int Gd\mu_\mathcal{O}=0$.
	Combining \eqref{b-1},  \eqref{b-12}, \eqref{b-5}, \eqref{b-6} and \eqref{b-7}, we have
	\begin{align*}&\ \ \ \int_{0}^{ m_2}G(\phi_t(z))dt\\
&\ge \int_{0}^{ m_3}G(\phi_t(z))dt+\int_{m_2-1}^{ m_2}G(\phi_t(z))dt\\
&= \int_{0}^{ m_3}G(\phi_t(z))-G(\phi_{t+v}(y_0))dt+\int_{0}^{ m_3}G(\phi_{t+v}(y_0))dt+\int_{m_2-1}^{ m_2}G(\phi_t(z))dt\\
&\ge -  \|h\|_1(1+\frac{\|\psi\|_1}{\psi_{min}}) \cdot \frac{4C^{3}\rho_2}{\lambda}-|\mathcal{O}|\cdot  \frac{1+\psi_{min}}{\psi_{min}}\|h\|_0
+\theta\left(\frac{D(\mathcal{O})}{4C^3e^{2\beta}}\right)- \frac{1+\psi_{min}}{\psi_{min}}\|h\|_0\\
&=\theta\left(\frac{D(\mathcal{O})}{4C^3e^{2\beta}}\right)- \left((1+\frac{\|\psi\|_1}{\psi_{min}}) \cdot \frac{4C^{3}\rho_2}{\lambda}+|\mathcal{O}|\cdot  \frac{1+\psi_{min}}{\psi_{min}}+\frac{1+\psi_{min}}{\psi_{min}}\right)\|h\|_1\\
&>0,
\end{align*}
	where we used assumption \eqref{b-0}. Therefore, $m=m_2$ is the time as required since $m_2\ge 1> \tau$ by \eqref{b-5}. This completes the  proof of Claim F2.

\section*{Acknowledgement}
At the end, we would like to express our gratitude to Tianyuan Mathematical Center in Southwest China, Sichuan University and Southwest Jiaotong University for their support and hospitality.


\end{document}